\documentclass{amsart}
\usepackage[latin1]{inputenc}
\usepackage{amsmath}
\usepackage{amsthm}
\usepackage{amsfonts}
\usepackage{mathrsfs}
\usepackage{amssymb}
\usepackage{stmaryrd}
\usepackage{comment}
\usepackage[all,cmtip]{xy}
\usepackage{graphicx}
\usepackage{tikz-cd}

\tikzset{
	symbol/.style={
		draw=none,
		every to/.append style={
			edge node={node [sloped, allow upside down, auto=false]{$#1$}}}
	}
}

\usepackage[width=15.00cm, height=22.00cm]{geometry}
\newcommand{\mx}{\mathfrak{X}}
\newcommand{\dr}{\mathbf{d}}
\newcommand{\ldr}[1]{{{\pounds}}_{#1}}

\usepackage{color}
\usepackage{enumerate}
\usepackage[perpage]{footmisc}
\usepackage{hyperref}
\hypersetup{
	colorlinks,
	citecolor=black,
	filecolor=black,
	linkcolor=black,
	urlcolor=black
}

\usepackage[numbers]{natbib}
\usepackage[titletoc]{appendix}

\allowdisplaybreaks

\title{Modules and representations up to homotopy of Lie $n$-algebroids}
\author{M.~Jotz Lean, R.~A.~Mehta, T.~Papantonis} \date{}

\theoremstyle{definition}
\newtheorem{theorem}{Theorem}[subsection]
\newtheorem{lemma}[theorem]{Lemma}
\newtheorem{proposition}[theorem]{Proposition}
\newtheorem{corollary}[theorem]{Corollary}
\newtheorem{definition}[theorem]{Definition}
\newtheorem*{question*}{Question}
\newtheorem*{conjecture*}{Conjecture}
\theoremstyle{remark}
\newtheorem*{remark*}{Remark}
\newtheorem{example}[theorem]{Example}
\newtheorem{remark}[theorem]{Remark}
\DeclareMathOperator{\Hom}{Hom}
\DeclareMathOperator{\End}{End}

\DeclareMathOperator{\sgn}{sgn}
\newcommand{\diff}{\mathrm{d}}
\DeclareMathOperator{\Der}{Der}
\DeclareMathOperator{\Jac}{Jac}
\DeclareMathOperator{\ad}{ad}
\DeclareMathOperator{\id}{id}

\newcommand{\cin}{\mathcal{C}^\infty} \newcommand{\Z}{\mathbb{Z}}
\newcommand{\R}{\mathbb{R}} \newcommand{\M}{\mathcal{M}}
 \newcommand{\Em}{\mathscr{E}}
\newcommand{\Fm}{\mathscr{F}} \newcommand{\N}{\mathcal{N}}
\newcommand{\Q}{\mathcal{Q}} \newcommand{\D}{\mathcal{D}}
\newcommand{\n}{[n]}

\newcommand{\E}{\underline{E}} \newcommand{\A}{\underline{A}}

\makeatletter
\newcommand\xleftrightarrow[2][]{%
	\ext@arrow 9999{\longleftrightarrowfill@}{#1}{#2}}
\newcommand\longleftrightarrowfill@{%
	\arrowfill@\leftarrow\relbar\rightarrow}
      \makeatother

\begin{document}

          \begin{abstract}
            This paper studies differential graded modules and
            representations up to homotopy of Lie $n$-algebroids, for
            general $n\in\mathbb{N}$. The adjoint and coadjoint
            modules are described, and the corresponding split
            versions of the adjoint and coadjoint representations up
            to homotopy are explained. In particular, the case of Lie
            2-algebroids is analysed in detail.  The compatibility of
            a Poisson bracket with the homological vector field of a
            Lie $n$-algebroid is shown to be equivalent to a morphism
            from the coadjoint module to the adjoint module, leading
            to an alternative characterisation of non-degeneracy of
            higher Poisson structures. Moreover, the Weil algebra of a
            Lie $n$-algebroid is computed explicitly in terms of
            splittings, and representations up to homotopy of Lie
            $n$-algebroids are used to encode decomposed VB-Lie
            $n$-algebroid structures on double vector bundles.
              \end{abstract}

              \maketitle
	
	\tableofcontents
	
	\section{Introduction}
	
	Lie $n$-algebroids, for $n\in\mathbb{N}$, are graded
        geometric structures which generalise the notion of
        Lie algebroids. They have become
        a field of much interest in mathematical physics, since they form a nice framework for
        higher analogues of Poisson and symplectic
        structures. 
        
        Courant algebroids \cite{LiWeXu97} give an important example
        of such higher structures. The work of Courant and Weinstein
        \cite{CoWe88} and of Hitchin and Gualtieri
        \cite{Hitchin03,Gualtieri03,Gualtieri07} show that Courant
        algebroids serve as a convenient framework for Hamiltonian
        systems with constraints, as well as for generalised geometry.
        A significant result from Roytenberg \cite{Roytenberg02} and
        \v{S}evera \cite{Severa05} shows that Courant algebroids are
        in one-to-one correspondence with Lie 2-algebroids which are
        equipped with a compatible symplectic structure.
	
	The standard super-geometric description of a Lie
        $n$-algebroid generalises the differential algebraic way of
        defining a usual Lie algebroid, as a vector bundle $A$ over a
        smooth manifold $M$ together with a degree 1 differential
        operator on the space
        $\Omega^\bullet(A):=\Gamma(\wedge^\bullet A^*)$. In the
        language of graded geometry, this is equivalent to a graded
        manifold of degree 1 equipped with a homological vector field
        \cite{Vaintrob97}, i.e.~a degree 1 derivation on its sheaf of
        functions which squares to zero and satisfies the graded
        Leibniz rule. A Lie $n$-algebroid is then defined as a graded
        manifold $\M$ of degree $n$, whose sheaf of functions
        $\cin(\M)$ is equipped with a homological vector field $\Q$.
          In more ``classical''
        geometric terms, a (split) Lie $n$-algebroid can also be defined as a
        graded vector bundle $\A=\bigoplus_{i=1}^n A_i$ over a smooth
        manifold $M$ together with some multi-brackets on its space of
        sections $\Gamma(\A)$ which satisfy some higher Leibniz and
        Jacobi identities \cite{ShZh17}. A Lie $n$-algebroid $(\M,\Q)$
        is called \emph{Poisson} if its underlying manifold carries a
        degree $-n$ Poisson structure $\{\cdot\,,\cdot\}$ on its sheaf
        of functions $\cin(\M)$, such that the homological vector
        field is a derivation of the Poisson bracket.

	A well-behaved representation theory of Lie $n$-algebroids for
        $n\geq2$ has not been developed yet. In the case $n=1$,
        i.e.~in the case of usual Lie algebroids, Gracia-Saz and Mehta
        \cite{GrMe10}, and independently Abad and Crainic
        \cite{ArCr12}, showed that the notion of
        \textit{representation up to homotopy} is a good notion of
        representation, which includes the adjoint
        representation. Roughly, the idea is to let the Lie algebroid
        act via a differential on Lie algebroid forms which take
        values on a cochain complex of vector bundles instead of just
        a single vector bundle. This notion is essentially a
        $\Z$-graded analogue of Quillen's super-representations
        \cite{Quillen85}. After their discovery, representations up to
        homotopy have been extensively studied in other works, see
        e.g.~\cite{Mehta09,ArCrDh11,ArSc11,ArSc13,DrJoOr15,Mehta15,TrZh16,GrJoMaMe18,
          BrCaOr18,Jotz19b,BrOr19,Vitagliano15b}. In particular, in
        \cite{Mehta09} it was shown that representations up to
        homotopy of Lie algebroids are equivalent, up to isomorphism,
        to Lie algebroid modules in the sense of \cite{Vaintrob97}.
	
	This paper extends the above notions of modules, and
        consequently of representations up to homotopy, to the context
        of higher Lie algebroids. The definition is the natural
        generalisation of the case of usual Lie algebroids explained
        above, i.e.~differential graded modules over the space of
        smooth functions of the underlying graded manifold.  The
        obtained notion is analysed in detail, including the two most
        important examples of representations, namely, \textit{the
          adjoint} and \textit{the coadjoint} representations (up to
        homotopy). An equivalent geometric point of view of a special
        class of representations is given by \emph{split VB-Lie
          $n$-algebroids}, i.e.~double vector bundles with a graded side and
        a linear split Lie $n$-algebroid structure over a split Lie $n$-algebroid.
   
        Our general motivation for studying representations up to
        homotopy of higher Lie $n$-algebroids comes from the case
        $n=2$, and in particular from Courant algebroids. More
        precisely, it is the search for a good notion not only of the
        adjoint representation of a Courant algebroid, but also of its
        ideals, similar to the work done in \cite{JoOr14}. The natural
        question that arises then is the following:
	
	\begin{question*}
		Is a compatible Poisson or symplectic structure on a Lie
		$n$-algebroid encoded in its adjoint representation?
	\end{question*}

        The answer to this question is positive, since it turns out
        that a Poisson bracket on a Lie $n$-algebroid gives rise to a
        natural map from the coadjoint to the adjoint
        representation which is an anti-morphism of left representations and a morphism of right representations (see Theorem \ref{thm_poisson}, Corollary
        \ref{cor_poisson} and Section
        \ref{morphism_of_ad*_ad_Poisson012}), i.e. it anti-commutes with the differentials of their structure as left representations and commutes with the differentials of their structure as right representations. Further, the Poisson
        structure is symplectic if and only if this map is in fact a
        left anti-isomorphism and right isomorphism. This result is already known in some special cases, including
        Poisson Lie $0$-algebroids, i.e.~ordinary Poisson manifolds
        $(M,\{\cdot\,,\cdot\})$, and Courant algebroids over a point,
        i.e.~quadratic Lie algebras
        $(\mathfrak{g},[\cdot\,,\cdot],\langle\cdot\,,\cdot\rangle)$. In
        the former case the map reduces to the natural map
        $\sharp\colon T^*M\to TM$ obtained from the Poisson bracket on
        $M$, and in the latter case it is the inverse of the map
        defined by the nondegenerate pairing
        $\mathfrak{g}\to\mathfrak{g^*},x\mapsto\langle x,\cdot
        \rangle$.
    
        \subsection*{Outline of the paper}

    This paper consists of seven sections and is organised as
    follows. Section \ref{prelim_sec} sets the notation and conventions,
    and recalls the definitions and constructions of graded vector
    bundles and Lie algebroids.
    
    Section \ref{Lie_n} offers a quick introduction to graded
    manifolds, (split) Lie $n$-algebroids, and Poisson and symplectic
    structures on Lie $n$-algebroids. In particular, it discusses the
    space of generalised functions of a Lie $n$-algebroid, gives the
    geometric description of a split Lie 2-algebroid \cite{Jotz19b}
    which is used in the rest of the paper, and defines the Weil
    algebra of a Lie $n$-algebroid -- as it is done in \cite{Mehta06}
    in the case $n=1$.
    
    Sections \ref{modules} and \ref{ruth} generalise the notions of
    Lie algebroid modules and representations up to homotopy to the
    setting of Lie $n$-algebroids. They offer a detailed explanation
    of the theory and give some useful examples, including the classes
    of the adjoint and coadjoint modules, whose properties are
    discussed thoroughly, especially in the case of Lie
    2-algebroids. Section \ref{modules} provides the answer to the
    question expressed above about the connection between higher
    Poisson or symplectic structures and the adjoint and coadjoint
    modules.
    
    Section \ref{Split VB-Lie n-algebroids} recalls some basic
    definitions and examples from the theory of double vector bundles
    and defines VB-Lie $n$-algebroids together with the prototype
    example of the tangent prolongation of a Lie $n$-algebroid. It
    also shows that there is a 1-1 correspondence between split VB-Lie
    $n$-algebroids and (left) representations up to homotopy of degree $n+1$,
    which relates again the adjoint representation of a Lie algebroid
    with its tangent prolongation.
    
    Finally, Section \ref{applications} discusses in the split case
    the results of this paper. It analyses the Weil algebra of a split
    Lie $n$-algebroid using vector bundles and connections, and it gives
    more details about the map between the coadjoint and adjoint
    representations for split Poisson Lie algebroids of degree
    $n\leq2$.

    \subsection*{Acknowledgements}
    The authors thank Miquel Cueca Ten, Luca Vitagliano and Chenchang
    Zhu for interesting discussions and remarks.  

    \subsection*{Relation to other work}
    During the preparation of this work, the authors learnt that
    Caseiro and Laurent-Gengoux also consider representations up to
    homotopy of Lie $n$-algebroids, in particular the adjoint
    representation, in their work in preparation \cite{CaLa19}.

    In \cite{Vitagliano15b}, Vitagliano considers representations of
    strongly homotopy of Lie-Rinehart algebras. Strongly homotopy Lie
    Rinehart algebras are the purely algebraic versions of graded
    vector bundles, over graded manifolds, equipped with a homological
    vector field that is tangent to the zero section. If the base
    manifold has grading concentrated in degree $0$ and the vector
    bundle is negatively graded, the notion recovers the one of split
    Lie $n$-algebroids. In that case, Vitagliano's representations
    correspond to the representations up to homotopy considered in
    this paper.

    In addition, since the DG $\mathcal M$-modules considered in this
    paper are the sheaves of sections of $\Q$-vector bundles, they are
    themselves also special cases of Vitagliano's strongly homotopy
    Lie-Rinehart algebras.

      \section{Preliminaries}\label{prelim_sec}
	
	This section recalls
    basic definitions and conventions that are used later on. In
    what follows, $M$ is a smooth manifold and all the considered
    objects are supposed to be smooth even if not explicitly
    mentioned. Moreover, all (graded) vector bundles are assumed
    to have finite rank.
	
	\subsection{(Graded) vector bundles and complexes}
	
	Given two ordinary vector bundles $E\to M$ and $F\to N$, there
        is a bijection between vector bundle morphisms $\phi\colon E\to F$
        covering $\phi_0\colon M\to N$ and morphisms of modules
        $\phi^\star\colon \Gamma(F^*)\to \Gamma(E^*)$
        over the pull-back
        $\phi_0^*\colon C^\infty(N)\to C^\infty(M)$. Explicitly, the
        map $\phi^\star$ is defined by
        $\phi^\star(f)(m)=\phi_m^*f_{\phi_0(m)}$, for
        $f\in\Gamma(F),m\in M$.
	
	Throughout the paper, underlined symbols denote
        graded objects. For instance, a graded vector bundle is a
        vector bundle $q\colon \underline{E}\to M$ together with a
        direct sum decomposition
	\[
	\underline{E}=\bigoplus_{n\in\mathbb{Z}}E_n[n]
	\]
	of vector bundles $E_n$ over $M$. The finiteness assumption
        for the rank of $\E$ implies that $\E$ is both upper and lower
        bounded, i.e. there exists a $n_0\in\mathbb{Z}$ such that
        $E_n=0$ for all $|n|>n_0$. Here, an element $e\in E_n$ is
        (degree-)homogeneous of degree $|e|=-n$. That is, for
        $k\in\Z$, the degree $k$ component of $\E$ (denoted with upper
        index $\E^k$) equals $E_{-k}$.
        
        All the usual algebraic constructions from the theory of ordinary
        vector bundles extend to the graded setting. More precisely,
        for graded vector bundles $\E,\underline{F}$, the dual $\E^*$,
        the direct sum $\E\oplus\underline{F}$, the space of graded
        homomorphisms $\underline{\Hom}(\E,\underline{F})$, the tensor
        product $\E\otimes\underline{F}$, and the symmetric and
        antisymmetric powers $\underline{S}(\E)$ and
        $\underline{A}(\E)$ are defined.

	A (cochain) complex of vector bundles is a graded vector
        bundle $\underline{E}$ over $M$ equipped with a degree one endomorphism
        (called the differential)
	\[
          \ldots\overset{\partial}{\to}E_{i+1}\overset{\partial}{\to}E_{i}\overset{\partial}{\to}E_{i-1}\overset{\partial}{\to}\ldots
	\]
	which squares to zero; $\partial^2=0$.
	
	Given two complexes $(\underline{E}, \partial)$ and
        $(\underline{F},\partial)$, one may construct new complexes by
        considering all the constructions that were discussed
        before. Namely, the bundles $\underline{S}(\underline{E})$,
        $\underline{A}(\underline{E})$, $\underline{E}^*$,
        $\underline{\Hom}(\underline{E},\underline{F})$ and
        $\underline{E}\otimes\underline{F}$ inherit a degree one
        operator that squares to 0. The basic principle for all the
        constructions is the graded derivation rule. For example, for
        $\phi\in\underline{\Hom}(\underline{E},\underline{F})$ and
        $e\in \underline{E}$:
	\[
	\partial(\phi(e)) = \partial(\phi)(e) + (-1)^{|\phi|}\phi(\partial(e)).	
	\]
	This can also be expressed using the language of (graded) commutators as
	\[
          \partial(\phi) = [\partial,\phi] = \partial\circ\phi -
          (-1)^{|\phi|}\phi\circ\partial=\partial\circ\phi -
          (-1)^{|\phi|\cdot|\partial|}\phi\circ\partial.
	\]
	
	The shift functor $[k]$, for $k\in\mathbb{Z}$, yields a new
        complex $(\E[k],\partial[k])$ whose $i$-th component is
        $\E[k]^i=\E^{i+k}=E_{-i-k}$ with differential
        $\partial[k]=\partial$. Formally, $\E[k]$ is obtained by
        tensoring with $(\mathbb{R}[k],0)$ from the right\footnote{If
          one chose to tensor from the left, the resulting complex
          would still have $i$-th component $\E[k]^i=\E^{i+k}$, but
          the Leibniz rule would give the differential
          $\partial[k] = (-1)^k \partial$.}. A degree $k$ morphism
        between two complexes $(\underline{E},\partial)$ and
        $(\underline{F},\partial)$ over $M$, or simply $k$-morphism,
        is, by definition, a degree preserving morphism
        $\phi\colon \underline{E}\to \underline{F}[k]$ over the
        identity on $M$; that is, a family of vector bundle maps
        $\phi_i\colon \E^i\to F[k]^i$ over the identity on $M$ that
        commutes with the differentials:\footnote{This becomes
          $\phi\circ\partial=(-1)^k\partial\circ\phi$ for the other
          convention.} $\phi\circ\partial=\partial\circ\phi$.
        
	\subsection{Dull algebroids vs Lie algebroids}
        A dull algebroid \cite{Jotz18a} is a vector bundle $Q\to M$
        endowed with an anchor $\rho_Q\colon Q\to TM$ and a bracket
        (i.e.~an $\mathbb R$-bilinear map)
        $[\cdot\,,\cdot]\colon\Gamma(Q)\times\Gamma(Q)\to\Gamma(Q)$ on
        its space of sections $\Gamma(Q)$, such that
          \begin{equation}\label{comp_anchor_bracket}
            \rho_Q[q_1,q_2] = [\rho_Q(q_1),\rho_Q(q_2)]
          \end{equation}
          and the Leibniz identity is satisfied in both entries:
		\[
                  [f_1q_1,f_2q_2] = f_1f_2[q_1,q_2] +
                  f_1\rho_Q(q_1)f_2\cdot q_2 - f_2\rho_Q(q_2)f_1\cdot
                  q_1,
		\]
		for all $q_1,q_2\in\Gamma(Q)$ and all $f_1,f_2\in
                C^\infty(M)$.

	A dull algebroid is a Lie algebroid if its
        bracket is also skew-symmetric and satisfies the Jacobi
        identity
	\[
	\Jac_{[\cdot\,,\cdot]}(q_1,q_2,q_3) := [q_1,[q_2,q_3]] - [[q_1,q_2],q_3] - [q_2, [q_1,q_3]] = 0,
	\]
	for all $q_1,q_2,q_3\in\Gamma(Q)$. 

    Given a skew-symmetric dull algebroid
    $Q$, there is an associated operator $\diff_Q$ of degree 1 on
    the space of $Q$-forms
    $\Omega^\bullet(Q) = \Gamma(\wedge^\bullet Q^*)$, defined by
    the formula
	\begin{align*}
	\diff_Q\tau(q_1,\ldots,q_{k+1}) = & \sum_{i<j}(-1)^{i+j}\tau([q_i,q_j],q_1,\ldots,\hat{q_i},\ldots,\hat{q_j},\ldots,q_{k+1}) \\
	& + \sum_i(-1)^{i+1}\rho_Q(q_i)(\tau(q_1,\ldots,\hat{q_i},\ldots,q_{k+1})),
	\end{align*}
	for $\tau\in\Omega^k(Q)$ and $q_1,\ldots,q_{k+1}\in\Gamma(Q)$; the notation $\hat{q}$ means that $q$ is omitted. The operator $\diff_Q$ satisfies as usual
	\[
	\diff_Q(\tau_1\wedge\tau_2) = (\diff_Q\tau_1)\wedge\tau_2 + (-1)^{|\tau_1|}\tau_1\wedge \diff_Q\tau_2,
	\]
	for $\tau_1,\tau_2\in\Omega^\bullet(Q)$. In general, the
        operator $\diff_Q$ squares to zero only on 0-forms
        $f\in \Omega^0(M)=C^\infty(M)$, since $\diff_Q^2f=0$ for all
        $f\in C^\infty(M)$ is equivalent to the compatibility of the
        anchor with the bracket \eqref{comp_anchor_bracket}. The
        identity $\diff_Q^2 = 0$ on all forms is equivalent to
        $(Q,\rho_Q,[\cdot\,,\cdot])$ being a Lie algebroid. 
	
	\subsection{Basic connections and basic curvature}
        Let $Q\to M$ be a skew-symmetric dull algebroid as in the last section and
        suppose that $E\to M$ is another vector bundle. A
        $Q$-connection on $E$ is defined similarly as usual, as a map
        $\nabla\colon\Gamma(Q)\times\Gamma(E)\to\Gamma(E),(q,e)\mapsto
        \nabla_q e$ that is $C^\infty(M)$-linear in the first
        argument and satisfies
	\[
	\nabla_q(fe) = \ldr{\rho_Q(q)}f\cdot e + f\nabla_q e,
	\]
	for all $q\in\Gamma(Q),e\in\Gamma(E)$ and $f\in
        C^\infty(M)$. The dual connection
        $\nabla^*$ is the $Q$-connection on $E^*$ defined by the
        formula
		\[
	\langle \nabla_q^* \varepsilon,e \rangle = \ldr{\rho_Q(q)}\langle \varepsilon,e \rangle - \langle \varepsilon,\nabla_q e \rangle,
	\]
	for all $\varepsilon\in\Gamma(E^*),e\in\Gamma(E)$ and
        $q\in\Gamma(Q)$, where $\langle\cdot\,,\cdot\rangle$ is the natural
        pairing between $E$ and its dual $E^*$.

        A $Q$-connection on a graded vector bundle
        $(\underline{E}=\bigoplus_{n\in\mathbb Z} E_n[n], \partial)$
        is a family of $Q$-connections $\nabla^n$, $n\in\mathbb N$, on
        each of the bundles $E_n$. If $\underline{E}$ is a complex
        with differential $\partial$, then the $Q$-connection is
        \emph{a connection on the complex $(\underline{E},\partial)$}
        if it commutes with $\partial$, i.e.,
        $\partial(\nabla^{n}_qe)=\nabla_q^{n-1}(\partial e)$ for
        $q\in\Gamma(Q)$ and $e\in \Gamma(E_n)$.
	
	The curvature of a $Q$-connection on a vector bundle $E$ is
        defined by
	\[
	R_\nabla(q_1,q_2)e = \nabla_{q_1}\nabla_{q_2} e - \nabla_{q_2}\nabla_{q_1} e - \nabla_{[q_1,q_2]} e,
	\] 
	for all $q_1,q_2\in\Gamma(Q)$ and $e\in\Gamma(E)$, and
        generally, it is an element of
        $\Gamma(Q^*\otimes Q^*\otimes E^*\otimes E)$. If the dull
        bracket of $Q$ is skew-symmetric, then the curvature is a
        2-form with values in the endomorphism bundle
        $\End(E)=E^*\otimes E$: $R_\nabla\in\Omega^2(Q,\End(E))$. A
        connection is called as usual \emph{flat} if its curvature
        $R_\nabla$ vanishes identically.
	
        Given a $Q$-connection $\nabla$ on $E$, and assuming that
        $[\cdot\,,\cdot]$ is skew-symmetric, there is an induced
        operator $\diff_\nabla$ on the space of $E$-valued $Q$-forms
        $\Omega^\bullet(Q,E) = \Omega^\bullet(Q)\otimes_{C^\infty(M)} \Gamma(E)$
        given by the usual Koszul formula
	\begin{align*}
	\diff_\nabla\tau(q_1,\ldots,q_{k+1}) = & \sum_{i<j}(-1)^{i+j}\tau([q_i,q_j],q_1,\ldots,\hat{q_i},\ldots,\hat{q_j},\ldots,q_{k+1}) \\
	& + \sum_i(-1)^{i+1}\nabla_{q_i}(\tau(q_1,\ldots,\hat{q_i},\ldots,q_{k+1})),
	\end{align*}
	for all $\tau\in\Omega^k(Q,E)$ and $q_1,\ldots,q_{k+1}\in\Gamma(Q)$. It satisfies
	\[
	\diff_\nabla(\tau_1\wedge\tau_2) = \diff_Q\tau_1\wedge\tau_2 + (-1)^{k}\tau_1\wedge \diff_\nabla\tau_2,
	\]
	for all $\tau_1\in\Omega^k(Q)$ and $\tau_2\in\Omega^\bullet(Q,E)$,
        and squares to zero if and only if $Q$ is a Lie algebroid and $\nabla$ is flat.
	
	Suppose that
        $\nabla\colon \mathfrak{X}(M)\times\Gamma(Q)\to\Gamma(Q)$ is a
        $TM$-connection on the vector bundle $Q$. The induced
        \textit{basic connections} on $Q$ and $TM$ are defined
        similarly as the ones associated to Lie algebroids \cite{GrMe10,ArCr12}:
	\[
          \nabla^{\text{bas}}=\nabla^{\text{bas},Q}\colon \Gamma(Q)\times\Gamma(Q)\to\Gamma(Q),\
          \nabla^{\text{bas}}_{q_1} q_2 = [q_1,q_2] +
          \nabla_{\rho_Q(q_2)} q_1
	\]
	and
	\[
          \nabla^{\text{bas}}=\nabla^{\text{bas},TM}\colon\Gamma(Q)\times\mathfrak{X}(M)\to\mathfrak{X}(M),\
          \nabla^{\text{bas}}_{q} X = [\rho_Q(q),X] + \rho_Q(\nabla_X
          q).
	\]
	The basic curvature is the form
        $R_\nabla^{\text{bas}}\in \Omega^2(Q, \Hom(TM,Q))$
        defined by
	\[
          R_\nabla^{\text{bas}}(q_1,q_2)X = -\nabla_X[q_1,q_2] +
          [q_1,\nabla_Xq_2] +[\nabla_Xq_1,q_2] +
          \nabla_{\nabla_{q_2}^\text{bas}X}q_1 -
          \nabla_{\nabla_{q_1}^\text{bas}X}q_2.
	\]
        The
        basic connections and the basic curvature satisfy
	\begin{equation}\label{eq_bas_Q_1}
	\nabla^{\text{bas},TM}\circ\rho_Q = \rho_Q\circ\nabla^{\text{bas},Q},
	\end{equation}
	\begin{equation}\label{eq_bas_Q_2}
	\rho_Q\circ R_\nabla^{\text{bas}} = R_{\nabla^{\text{bas},TM}},
	\end{equation}
	\begin{equation}\label{eq_bas_Q_3}
          R_\nabla^{\text{bas}}\circ\rho_Q + \Jac_{[\cdot\,,\cdot]} = R_{\nabla^{\text{bas},Q}}.
	\end{equation}

	\section{(Split) Lie $n$-algebroids and $\mathbb{N}\Q$-manifolds}\label{Lie_n}

	This section recalls basic results about
        $\mathbb{N}$-manifolds and Lie $n$-algebroids (based on
        \cite{Jotz18b}), and describes the Weil algebra of a Lie
        $n$-algebroid for general $n$ (see \cite{Mehta09} for
        $n=1$). It focusses on the category of \emph{split
          $\mathbb{N}$-manifolds}, which is isomorphic modulo some
        choices of splittings to the category of
        $\mathbb{N}$-manifolds (\cite{BoPo13, Roytenberg02}). 
	
	\subsection{(Split) $\mathbb{N}$-manifolds and homological vector fields}\label{lie2lie3}
	
	Graded manifolds of degree $n\in\mathbb{N}$ are defined as
        follows, in terms of sheaves over ordinary smooth manifolds.
	
	\begin{sloppypar}
		\begin{definition}
                  An \emph{$\mathbb{N}$-manifold $\M$ of degree $n$ and
                  dimension $(m;r_1,\ldots,r_n)$} is a sheaf
                  $\cin(\M)$ of $\mathbb{N}$-graded,
                  graded commutative, associative, unital
                  $C^\infty(M)$-algebras over a smooth $m$-dimensional
                  manifold $M$, which is locally freely generated by
                  $r_1+\ldots+r_n$ elements
                  $\xi_1^1,\ldots,\xi_1^{r_1}, \xi_2^1,\ldots,
                  \xi_2^{r_2},\ldots, \xi_n^1,\ldots,\xi_n^{r_n}$ with
                  $\xi_i^j$ of degree $i$ for $i\in\{1,\ldots,n\}$ and
                  $j\in \{1,\ldots,r_i\}$.
			
                  A morphism of $\mathbb{N}$-manifolds $\mu\colon \N\to\M$
                  over a smooth map $\mu_0\colon N\to M$ of the underlying
                  smooth manifolds is a morphism of sheaves of graded
                  algebras $\mu^\star\colon \cin(\M)\to \cin(\N)$
                  over $\mu_0^\ast\colon C^\infty(M)\to C^\infty(N)$.
		\end{definition}
	\end{sloppypar}
	
	For short, ``$\n$-manifold'' means ``$\mathbb{N}$-manifold of
        degree $n$''. The degree of a (degree-)homogeneous element
        $\xi\in \cin(\M)$ is written $|\xi|$. Note that the degree 0
        elements of $\cin(\M)$ are just the smooth functions of the
        manifold $M$. By definition, a $[1]$-manifold $\M$ is a
        locally free and finitely generated sheaf $ \cin(\M)$ of
        $C^\infty(M)$-modules. That is, $\cin(\M)=\Gamma(\wedge E^*)$
        for a vector bundle $E\to M$. In that case, $\M=:E[1]$. \emph{Recall
        that this means that the elements of $E$ have degree $-1$, and
        so the sections of $E^*$ have degree $1$.}
	
      Consider now a (non-graded) vector bundle $E$ of rank $r$
      over the smooth manifold $M$ of dimension $m$. Similarly as before, assigning the
      degree $n$ to the fibre coordinates of $E$ defines an
      $[n]$-manifold of dimension $(m;r_1=0,\ldots,r_{n-1}=0,r_n=r)$
      denoted by $E[n]$, with $\cin(E[n])^n=\Gamma(E^*)$. More
      generally, let $E_1,\ldots,E_n$ be vector bundles of ranks
      $r_1,\ldots,r_n$, respectively, and assign the degree $i$ to the
      fibre coordinates of $E_i$, for each $i=1,\ldots,n$. The direct
      sum $\E=E_1[1]\oplus\ldots\oplus E_n[n]$ is a graded vector
      bundle with grading concentrated in degrees $-1,\ldots,-n$. When
      seen as an $[n]$-manifold, $E_1[1]\oplus\ldots\oplus E_n[n]$ has
      the local basis of sections of $E_i^*$ as local generators of
      degree $i$ and thus its dimension is $(m;r_1,\ldots,r_n)$.

	\begin{definition}
          An $[n]$-manifold of the form
          $E_1[1]\oplus\ldots\oplus E_n[n]$ as above is called a
          \emph{split $[n]$-manifold}.
	\end{definition}
	
	The relation between $[n]$-manifolds and split $[n]$-manifolds
        is explained by the following theorem, which is implicit in
        \cite{Roytenberg02} and explicitly proved in \cite{BoPo13}.
	\begin{theorem}
		Any $[n]$-manifold is non-canonically diffeomorphic to a split $[n]$-manifold.
	\end{theorem}
        Note that under the above correspondence, the structure sheaf
        of an [$n$]-manifold
        $\M \simeq \underline{E} = E_1[1]\oplus\ldots\oplus E_n[n]$ becomes
\[
  \cin(\M) \simeq \Gamma(\underline{S}(\underline{E}^*)),
\]
and a different choice of splitting leaves the bundles unchanged. In
particular, for the case of a split [2]-manifold
$\M=\underline{E}=E_1[1]\oplus E_2[2]$ the graded functions are
\[
\cin(\M)=\Gamma(\underline{S}(\underline{E}^*))= \Gamma(\wedge E^*_1\otimes SE^*_2),
\]
where the grading is defined such that
\[
  \cin(\M)^i=\bigoplus_{k+2\ell=i}\Gamma\left(\wedge^kE^*_1\otimes
    S^\ell E^*_2\right).
\]

Using the language of graded derivations, the usual notion of  vector
field can be generalized to a notion of vector field on an $[n]$-manifold
$\M$.
\begin{definition}
  A vector field of degree $j$ on $\M$ is a degree $j$ (graded)
  derivation of $\cin(\M)$, i.e. a map
  $\mathcal{X}\colon\cin(\M)\to \cin(\M)$ such that
  $|\mathcal{X}(\xi)|=j+|\xi|$ and
  $\mathcal{X}(\xi\zeta)=\mathcal{X}(\xi)\zeta+(-1)^{j|\xi|}\xi
  \mathcal{X}(\zeta)$, for homogeneous elements
  $\xi,\zeta\in C^\infty(\M)$.
      \end{definition}
  As usual, $|\mathcal{X}|$ denotes the degree of a homogeneous vector field $\mathcal{X}$. The Lie bracket of two vector fields
  $\mathcal{X},\mathcal{Y}$ on $\M$ is the graded commutator
  \[
  [\mathcal{X},\mathcal{Y}]=\mathcal{X}\mathcal{Y}-(-1)^{|\mathcal{X}||\mathcal{Y}|}\mathcal{Y}\mathcal{X}.
  \]
The following relations hold:
\begin{enumerate}[(i)]
	\item $[\mathcal{X},\mathcal{Y}]=-(-1)^{|\mathcal{X}||\mathcal{Y}|}[\mathcal{Y},\mathcal{X}]$,
	\item
          $[\mathcal{X},\xi
          \mathcal{Y}]=\mathcal{X}(\xi)\mathcal{Y}+(-1)^{|\mathcal{X}||\xi|}\xi[\mathcal{X},\mathcal{Y}]$,
	\item $(-1)^{|\mathcal{X}||\mathcal{Z}|}[\mathcal{X},[\mathcal{Y},\mathcal{Z}]]
	+(-1)^{|\mathcal{Y}||\mathcal{X}|}[\mathcal{Y},[\mathcal{Z},\mathcal{X}]]
	+(-1)^{|\mathcal{Z}||\mathcal{Y}|}[\mathcal{Z},[\mathcal{X},Y]]=0$,
\end{enumerate}
for $\mathcal{X},\mathcal{Y},\mathcal{Z}$ homogeneous vector fields on $\M$, and $\xi$
a homogeneous element of $\cin(\M)$.
  
Local generators $\xi_i^j$ of $\cin(\M)$ over an open set
$U\subseteq M$ given by the definition of $\M$ define the (local)
vector fields $\partial_{\xi_i^j}$ of degree $-j$, which sends
$\xi_i^j$ to $1$ and the other local generators to $0$. The sheaf
$\underline{\Der}_U(\cin(\M))$ of graded derivations of $\cin_U(\M)$ is
freely generated as a $C^\infty_U(\M)$-module by $\partial_{x_k}$
and $\partial_{\xi_i^j}$, where $x_1,\ldots,x_m$ are coordinates for
$M$ defined on $U$.
  
Note that in the case of a split $[n]$-manifold
$E_1[1]\oplus\ldots\oplus E_n[n]$, each section $e\in\Gamma(E_j)$
defines a derivation $\hat{e}$ of degree $-j$ on $\M$ by the
relations: $\hat{e}(f)=0$ for $f\in\cin(M)$,
$\hat{e}(\varepsilon)=\langle\varepsilon,e\rangle$ for
$\varepsilon\in\Gamma(E_j^*)$ and $\hat{e}(\varepsilon)=0$ for
$\varepsilon\in\Gamma(\E^*)$ with $|\varepsilon|\neq j$. In
particular, $\widehat{e_j^i}=\partial_{\varepsilon_j^i}$ for
$\{e_j^i\}$ a local basis of $E_j$ and $\{\varepsilon_j^i\}$ the dual
basis of $E_j^*$.

Given $TM$-connections $\nabla^i\colon\mathfrak{X}(M)\to\Der(E_i)$ for all
$i$, the space of vector fields over a split $[n]$-manifold $\M$ is generated as a
$\cin(\M)$-module by
\[
  \{ \nabla^1_X \oplus \ldots \oplus \nabla^n_X\ |\
  X\in\mathfrak{X}(M) \}\cup\{ \hat{e}\ |\ e\in\Gamma(E_i)\ \text{for
    some}\ i \}.
\]
The vector fields of the form
$\nabla^1_X \oplus \ldots \oplus \nabla^n_X$ are of degree 0 and are
understood to send $f\in C^\infty(M)$ to $X(f)\in C^\infty(M)$, and
$\varepsilon\in\Gamma(E_i^*)$ to
$\nabla_X^{i,*}\varepsilon\in\Gamma(E_i^*)$. The negative degree vector
fields are generated by those of the form $\hat{e}$.
\begin{definition}
  A homological vector field $\Q$ on an $[n]$-manifold $\M$ is a degree 1 derivation of $\cin(\M)$ such that
  $\Q^2=\frac{1}{2}[\Q,\Q]=0$.
\end{definition}
A homological vector field on a $[1]$-manifold $\M=E[1]$ is a
differential $\diff_E$ associated to a Lie algebroid structure on the
vector bundle $E$ over $M$ \cite{Vaintrob97}. The following definition
generalizes this to arbitrary degrees.

\begin{definition}\label{abstract_Lie_algebroids}
  A \emph{Lie $n$-algebroid} is an $[n]$-manifold $\M$ endowed with a
  homological vector field $\Q$ -- the pair $(\M, \Q)$ is also called
  \emph{$\mathbb{N}\Q$-manifold of degree $n$}. A \emph{split Lie
    $n$-algebroid} is a split $[n]$-manifold $\M$ endowed with a
    homological vector field $\Q$. A \emph{morphism of (split) Lie
    $n$-algebroids} is a morphism $\mu$ of the underlying [$n$]-manifolds
    such that $\mu^\star$ commutes with the homological vector fields.
\end{definition}

The homological vector field associated to a split Lie $n$-algebroid
$\A=A_1[1]\oplus\ldots\oplus A_n[n]\to M$ can be equivalently
described by a family of brackets which satisfy some Leibniz and
higher Jacobi identities \cite{ShZh17}. More precisely, a homological
vector field on $\A$ is equivalent to an $L_\infty$-algebra
structure\footnote{We note that the sign convention agrees with,
  e.g. \cite{KaSt06,Voronov05}. In \cite{Schatz09}, the term
  ``$L_\infty[1]$-algebra'' was used for brackets with this sign
  convention.}  on $\Gamma(\A)$ that is anchored by a vector bundle
morphism $\rho\colon A_1\to TM$.  Such a structure is given by
multibrackets
$\llbracket\cdot\,,\ldots,\cdot\rrbracket_i\colon\Gamma(\A)^i\to
\Gamma(\A)$ of degree $1$ for $1\leq i \leq n+1$ such that
 \begin{enumerate}
   \item $\llbracket\cdot\,,\cdot\rrbracket_2$ satisfies the Leibniz identity with
     respect to $\rho$,
     \item $\llbracket\cdot\,,\ldots,\cdot\rrbracket_i$ is $C^\infty(M)$-linear in each entry
       for all $i\neq2$,
     \item \emph{(graded skew symmetry)} each
       $\llbracket\cdot,\ldots,\cdot\rrbracket_i$ is graded
       alternating: for a permutation $\sigma\in S_i$ and for all
       $a_1,\ldots,a_i\in\Gamma(\A)$ degree-homogeneous sections
       \[\llbracket a_{\sigma(1)}, a_{\sigma(2)},\ldots,a_{\sigma(k)}\rrbracket_i
       =\text{Ksgn}(\sigma,a_1,\ldots,a_k)\cdot \llbracket a_1,a_2,\ldots,a_k\rrbracket_i, \] and
   \item \emph{(strong homotopy Jacobi identity)} for $k\in \mathbb N$
     and $a_1,\ldots, a_k\in\Gamma(\A)$ sections of homogeneous
     degree:
\[
  \sum_{i+j=k+1}(-1)^{i(j-1)}\sum_{\sigma\in\text{Sh}_{i,k-i}}
  \text{Ksgn}(\sigma,a_1,\ldots,a_k)\llbracket \llbracket
  a_{\sigma(1)},\ldots,a_{\sigma(i)} \rrbracket_i,
  a_{\sigma(i+1)},\ldots,a_{\sigma(k)} \rrbracket_j = 0.
\]
\end{enumerate}
Here, $\text{Sh}_{i,k-i}$ is the set of all
$(i,k-i)$-shuffles\footnote{A $(i,k-i)$-shuffle is an element
  $\sigma\in S_k$ such that $\sigma(1)<\ldots<\sigma(i)$ and
  $\sigma(i+1)<\ldots<\sigma(k)$.  } and
$\text{Ksgn}(\sigma,a_1,\ldots,a_k)$ is the $(a_1,\dots,a_k)$-graded
signature of the permutation $\sigma\in S_k$, i.e.
\[
  a_1\wedge\ldots\wedge a_k = \text{Ksgn}(\sigma, a_1,
  \ldots,a_k)a_{\sigma(1)}\wedge\ldots\wedge a_{\sigma(k)}.
\]

\medskip This gives the following alternative geometric description of
a split Lie $2$-algebroid
$(\mathcal M=A_1[1]\oplus A_2[2],\mathcal Q)$, see \cite{Jotz19b}.
For consistency with the notation in \cite{Jotz19b}, set $A_1:=Q$ and
$A_2^*=:B$.

\begin{definition}\label{geom_split_2-alg}
  A split Lie 2-algebroid $Q[1]\oplus B^*[2]$ is given by a pair of
  an anchored vector bundle $(Q\to M,\rho_Q)$ and a vector bundle
  $B\to M$, together with a vector bundle map $\ell\colon B^*\to Q$, a
  skew-symmetric dull bracket
  $[\cdot\,,\cdot]\colon \Gamma(Q)\times\Gamma(Q)\to \Gamma(Q)$, a
  linear $Q$-connection $\nabla$ on $B$, and a vector valued 3-form
  $\omega\in\Omega^3(Q,B^*)$ such that
	\begin{enumerate}[(i)]
		\item $\nabla^*_{\ell(\beta_1)}\beta_2 + \nabla^*_{\ell(\beta_2)}\beta_1 = 0$, for all $\beta_1,\beta_2\in\Gamma(B^*)$,
		\item $[q,\ell(\beta)]=\ell(\nabla_q^*\beta)$ for all
                  $q\in\Gamma(Q)$ and $\beta\in\Gamma(B^*)$,
		\item $\Jac_{[\cdot\,,\cdot]} = \ell\circ\omega\in\Omega^3(Q,Q)$,
		\item $R_{\nabla^*}(q_1,q_2)\beta = -\omega(q_1,q_2,\ell(\beta))$ for $q_1,q_2\in\Gamma(Q)$ and $\beta\in\Gamma(B^*)$,
		\item $\diff_{\nabla^*}\omega = 0$.
	\end{enumerate}  
      \end{definition}
      To pass
      from the definition above to the homological vector field $\Q$,
      set $\Q(f)=\rho^*\diff f \in\Gamma(Q^*)$,
      $\Q(\tau)=\diff_{Q}\tau+\partial_B\tau \in \Omega^2(Q)\oplus
      \Gamma(B)$, and
      $\Q(b)=\diff_{\nabla}b - \langle\omega, b\rangle \in \Omega^1(Q,
      B)\oplus \Omega^3(Q)$ for $f\in C^\infty(M),\tau\in\Omega(Q)$
      and $b\in\Gamma(B)$, where $\partial_B:=\ell^*$.

      On the other hand we may obtain the data of Definition
      \ref{geom_split_2-alg} from a given homological vector field
      $\Q$ as follows. Define the vector bundle map $\ell$ to be the 1-bracket and
      $\rho$ to be the anchor. The 2-bracket induces the skew-symmetric dull bracket
      on $Q$ and the $Q$-connection on $B^*$ via the formula
\[
  \llbracket q_1\oplus\beta_1,q_2\oplus\beta_2\rrbracket_2 =
  [q_1,q_2]_Q\oplus(\nabla_{q_1}^*\beta_2 - \nabla_{q_2}^*\beta_1).
\]
Finally, the 3-bracket induces the 3-form $\omega$ via the formula
\[
\llbracket q_1\oplus0,q_2\oplus0,q_3\oplus0\rrbracket_3 = 0\oplus\omega(q_1,q_2,q_3). 
\]

\begin{example}[Lie 2-algebras]
  If we consider a Lie 2-algebroid over a point, then we recover the
  notion of \emph{Lie 2-algebra} \cite{BaCr04}. Specifically, a Lie
  2-algebroid over a point consists of a pair of vector spaces
  $\mathfrak{g}_0,\mathfrak{g}_1$, a linear map
  $\ell\colon\mathfrak{g}_0\to\mathfrak{g}_1$, a skew-symmetric
  bilinear bracket
  $[\cdot\,,\cdot]\colon
  \mathfrak{g}_1\times\mathfrak{g}_1\to\mathfrak{g}_1$, a bilinear
  \textit{action bracket}
  $[\cdot\,,\cdot]\colon\mathfrak{g}_1\times\mathfrak{g}_0\to\mathfrak{g}_0$,
  and an alternating trilinear bracket
  $[\cdot\,,\cdot\,,\cdot]\colon\mathfrak{g}_1\times\mathfrak{g}_1\times\mathfrak{g}_1\to\mathfrak{g}_0$
  such that
  \begin{enumerate}
  \item $[\ell(x),y] + [\ell(y),x] = 0$ for $x,y\in\mathfrak{g}_0$,
  \item $[x,\ell(y)] = \ell([x,y])$ for $x\in\mathfrak{g}_1$ and $y\in\mathfrak{g}_0$,
  \item $\Jac_{[\cdot,\cdot]}(x,y,z) = \ell([x,y,z])$ for
    $x,y,z\in\mathfrak{g}_1$,
  \item $[[x,y],z] +[y,[x,z]] -[x,[y,z]] = [x,y,\ell(z)]$ for
    $x,y\in\mathfrak{g}_1$ and $z\in\mathfrak{g}_0$,
\item and the higher Jacobi identity
	\begin{align*}
          0 = & [x,[y,z,w]] - [y,[x,z,w]]+[z,[x,y,w]]- [w,[x,y,z]]\\ 
&  - [[x,y],z,w] + [[x,z],y,w] - [[x,w],y,z]- [[y,z],x,w] + [[y,w],x,z] - [[z,w],x,y].
	\end{align*}
	holds for $x,y,z,w\in\mathfrak{g}_1$.
\end{enumerate}
\end{example}

\begin{example}[Derivation Lie 2-algebr(oid)]
  For any Lie algebra $(\mathfrak{g},[\cdot\,,\cdot]_\mathfrak{g})$, the
  derivation Lie 2-algebra is defined as the complex
	\[
	\ad\colon \mathfrak{g}\to\Der(\mathfrak{g})
	\]
	with brackets given by
        $[\delta_1,\delta_2] = \delta_1\delta_2 - \delta_2\delta_1$,
        $[\delta,x] = \delta x$, $[\delta_1,\delta_2,\delta_3] = 0$
        for all $\delta,\delta_i\in\Der(\mathfrak{g}),i=1,2,3$, and
        $x\in\mathfrak{g}$.

        A global analogue of this construction can be achieved only
        under strong assumptions on the Lie algebroid $A\to
        M$. Precisely, let $A\to M$ be a Lie algebra bundle. Then the
        space of all derivations $D$ of the vector bundle $A$ which
        preserve the bracket
	\[
	D[a_1,a_2] = [Da_1,a_2] + [a_1,Da_2]
	\]
	is the module of sections of a vector bundle over $M$, denoted
        $\Der_{[\cdot,\cdot]}(A)\to M$. Together with the usual
        commutator bracket and the anchor $\rho'(D)=X$, where $D$ is a
        derivation of $\Gamma(A)$ covering $X\in\mathfrak{X}(M)$, the
        vector bundle $\Der_{[\cdot,\cdot]}(A)$ is a Lie algebroid
        over $M$ \cite{Mackenzie05}. Since the anchor of $A$
        is trivial, the complex
	\[
	A\overset{\ad}{\to}\Der_{[\cdot\,,\cdot]}(A)\overset{\rho'}{\to}TM
	\]
	becomes a Lie 2-algebroid with
        $\Der_{[\cdot,\cdot]}(A)$-connection on $A$ given by
        $\nabla_Da = Da$ and $\omega=0$.
\end{example}

\begin{example}[Courant algebroids]\label{Split_symplectic_Lie_2-algebroid_example}
  Let    $E\to   M$    be   a    Courant   algebroid    with   pairing
  $\langle\cdot\,,\cdot\rangle\colon E\times_M E\to  E$, anchor $\rho$
  and Dorfman bracket  $\llbracket\cdot\,,\cdot\rrbracket$, and choose
  a               metric               linear               connection
  $\nabla\colon   \mathfrak{X}(M)\times\Gamma(E)\to\Gamma(E)$.    Then
  $E[1]\oplus T^*M[2]$  becomes as follows a  split Lie $2$-algebroid.
  The     skew-symmetric     dull      bracket     is     given     by
  $[e,e'] =  \llbracket e,e'  \rrbracket -  \rho^*\langle \nabla_.e,e'
  \rangle$ for all $e,e'\in\Gamma(E)$.  The \emph{basic connection} is
  $\nabla^\text{bas}\colon\Gamma(E)\times\mathfrak{X}(M)\to\mathfrak{X}(M)$
  and     the     \textit{basic     curvature}     is     given     by
  $\omega_\nabla\in\Omega^2(E,\Hom(TM,E))$
	\[
          \omega_\nabla(e,e')X = -\nabla_X\llbracket e,e' \rrbracket +
          \llbracket \nabla_Xe,e' \rrbracket + \llbracket e,\nabla_Xe'
          \rrbracket + \nabla_{\nabla_{e'}^{\text{bas}} X} e -
          \nabla_{\nabla_{e}^{\text{bas}} X} e' - P^{-1}\langle
          \nabla_{\nabla_{.}^{\text{bas}}X}e,e' \rangle,
	\]
	where $P\colon E\to E^*$ is the isomorphism defined by the
        pairing, for all $e,e'\in\Gamma(E)$ and
        $X\in\mathfrak{X}(M)$. The map $\ell$ is
        $\rho^*\colon T^*M\to E$, the $E$-connection on $T^*M$ is
        $\nabla^{\text{bas},*}$ and the form
        $\omega\in\Omega^3(E,T^*M)$ is given by
        $\omega(e_1,e_2,e_3)=\langle \omega_\nabla(e_1,e_2)(\cdot),e_3
        \rangle$. The kind of split Lie 2-algebroids that arise in
        this way are the split symplectic Lie 2-algebroids
        \cite{Roytenberg02}. They are splittings of the symplectic Lie
        $2$-algebroid which is equivalent to the tangent prolongation
        of $E$, which is an LA-Courant algebroid
        \cite{Jotz19b,Jotz18d}.
      \end{example}

\subsection{Generalized functions of a Lie $n$-algebroid}
In the following, $(\M,\Q)$ is a Lie $n$-algebroid with underlying
manifold $M$. Consider the space
$\cin(\M)\otimes_{C^\infty(M)}\Gamma(\E)$ for a graded vector bundle
$\E\to M$ of finite rank. For simplicity,
$\cin(\M)\otimes_{C^\infty(M)}\Gamma(\E)$ is sometimes written
$\cin(\M)\otimes\Gamma(\E)$. That is, these tensor products in the
rest of the paper are always of $C^\infty(M)$-modules.

First suppose that $(\M,\Q)=(A[1],\diff_A)$ is a Lie algebroid. The
space of $\E$-valued differential forms
$\Omega(A;\E):=\Omega(A)\otimes_{C^\infty(M)}\Gamma(\E)=\cin(A[1])\otimes_{C^\infty(M)}\Gamma(\E)$
has a natural grading given by
\[
\Omega(A;\E)_p=\bigoplus_{i-j=p}\Omega^i(A;E_j).
\]
It is well-known (see \cite{ArCr12}) that any degree preserving vector
bundle map $h\colon \E\otimes \underline{F}\to \underline{G}$ induces a wedge product operation
\[
  (\cdot\wedge_h\cdot)\colon\Omega(A;\E)\times\Omega(A;\underline{F})\to
  \Omega(A;\underline{G})
\]
which is defined on  $\omega\in\Omega^p(A;E_i)$ and $\eta\in\Omega^q(A;F_j)$ by
\[
  (\omega\wedge_h\eta)(a_1,\ldots,a_{p+q})=\sum_{\sigma\in
    \text{Sh}_{p,q}}(-1)^{qi}\sgn(\sigma)h\left(\omega(a_{\sigma(1)},\ldots,a_{\sigma(p)}),\eta(a_{\sigma(p+1)},\ldots,a_{\sigma(p+q)})\right),
\]
for all $a_1,\ldots,a_{p+q}\in\Gamma(A)$. 

	In particular, the above rule reads
	\[
	\theta\wedge_h\zeta=(-1)^{qi}\left(\omega\wedge\eta\right)\otimes h(e,f),
	\]
	for all $\theta=\omega\otimes e$ and $\zeta=\eta\otimes f$ where
        $\omega$ is a $p$-form, $\eta$ is a $q$-form, and $e$ and $f$
        are homogeneous sections of $\E$ and $\underline{F}$ of degree $i$ and $j$,
        respectively.

        Some notable cases for special choices of the map $h$ are
        given by the identity, the composition of endomorphisms, the
        evaluation and the `twisted' evaluation maps, the graded
        commutator of endomorphisms and the natural pairing of a
        graded vector bundle with its dual. In particular, the evaluation $(\Phi,e)\mapsto \Phi(e)$ and the twisted
        evaluation $(e,\Phi)\mapsto(-1)^{|\Phi||e|}\Phi(e)$ make $\Omega(A;\E)$ a graded
        $\Omega(A;\underline{\End}(\E))$-bimodule.

In the general case of a Lie $n$-algebroid $(\M,\Q)$, the space $\Omega(A)$ is
replaced by the generalized smooth functions $\cin(\M)$ of
$\M$. The space $\cin(\M)\otimes_{C^\infty(M)}\Gamma(\E)$ has a natural grading,
where the homogeneous elements of degree $p$ are given by
\[
\bigoplus_{i-j=p}\cin(\M)^i\otimes\Gamma(E_j).
\]

Similarly as in the case of a Lie algebroid, given a degree preserving map
\[
h\colon \E\otimes \underline{F}\to \underline{G},
\]
one obtains the multiplication
\begin{align*}
  \left(\cin(\M)\otimes\Gamma(\E)\right)\times \left(\cin(\M)\otimes\Gamma(\underline{F})\right)\to\, & \cin(\M)\otimes\Gamma(\underline{G})\\
  (\omega,\eta)\mapsto\, & \omega\wedge_h\eta.
\end{align*}
In particular, for elements of the form
$\xi\otimes e\in \cin(\M)^i\otimes\Gamma(E_j),\zeta\otimes f\in
\cin(\M)^k\otimes\Gamma(F_\ell)$ the above rule reads
\[
\left(\xi\otimes e\right)\wedge_h\left(\zeta\otimes f\right)=(-1)^{(-j)k}\xi\zeta\otimes h(e,f),
\]
where on the right hand side the multiplication $\xi\zeta$ is the one
in $\cin(\M)$. The special cases above are defined similarly for the
$n$-algebroid case. Moreover,
$\cin(\M)\otimes_{C^\infty(M)}\Gamma(\E)$ is endowed with the
structure of a graded
$\cin(\M)\otimes_{C^\infty(M)}\Gamma(\underline{\End}(\E))$-bimodule.

Finally, the following lemma will be useful later as it is a
generalisation of \cite[Lemma A.1]{ArCr12}, and gives the connection
between the space
$\cin(\M)\otimes\Gamma(\underline{\Hom}(\E,\underline{F}))$ and the
homomorphisms from $\cin(\M)\otimes\Gamma(\E)$ to
$\cin(\M)\otimes\Gamma(\underline{F})$.
\begin{lemma}\label{wedge_product-operators_Correspondence_Lemma}
  There is a 1-1 correspondence between the degree $n$ elements of
  $\cin(\M)\otimes\Gamma(\underline{\Hom}(\E,\underline{F}))$ and the operators
  $\Psi\colon \cin(\M)\otimes\Gamma(\E)\to \cin(\M)\otimes\Gamma(\underline{F})$
  of degree $n$ which are $\cin(\M)$-linear in the graded sense:
	\[
	\Psi(\xi\wedge\eta)=(-1)^{nk}\xi\wedge \Psi(\eta),
	\]
	for all $\xi\in \cin(\M)^k$, and all $\eta\in \cin(\M)\otimes\Gamma(\E)$.
\end{lemma}
\begin{proof}
  The element $\Phi\in \cin(\M)\otimes\Gamma(\underline{\End}(\E))$
  induces the operator $\widehat{\Phi}$ given by left multiplication
  by $\Phi$:
	\[
	\widehat{\Phi}(\eta)=\Phi\wedge\eta.
	\]
	This clearly satisfies
        $\widehat{\Phi}(\xi\wedge\eta)=(-1)^{nk}\xi\wedge\widehat{\Phi}(\eta)$,
        for all
        $\xi\in \cin(\M)^k,\ \eta\in
        \cin(\M)\otimes\Gamma(\E)$.  Conversely, an operator $\Psi$ of degree $n$
        must send a section $e\in\Gamma(E_k)$ into the sum
	\[
          \Gamma(F_{k-n}) \oplus \left(
            \cin(\M)^1\otimes\Gamma(F_{k-n+1}) \right) \oplus
          \left( \cin(\M)^2\otimes\Gamma(F_{k-n+2}) \right)
          \oplus\dots,
	\]
	defining the elements
	\[
	\Psi_i\in C^\infty(\M)^i\otimes\Gamma(\underline{\Hom}^{n-i}(\E,\underline{F})).
	\]
	Thus, this yields the (finite) sum
        $\widetilde{\Psi} = \sum \Psi_i\in
        \Big(\cin(\M)\otimes\Gamma(\underline{\Hom}(\E,\underline{F}))\Big)^n$.
        Clearly,
	\[
	\widetilde{\widehat{\Phi}} = \Phi\ \text{and}\ \widehat{\widetilde{\Psi}} = \Psi.\qedhere
	\]
\end{proof}
Schematically, for a Lie $n$-algebroid $\M$, the above lemma gives the following diagram:
\[
\Big(\cin(\M)\otimes\Gamma(\underline{\Hom}(\E,\underline{F}))\Big)^n\xleftrightarrow{\text{1-1}}
\left\{\begin{array}{c}
\text{Degree}\ n\ \text{operators}\ \Psi \\
\cin(\M)\otimes\Gamma(\E)\to \cin(\M)\otimes\Gamma(\underline{F}) \\
\text{which are}\ \cin(\M)\text{-linear in the graded sense}
\end{array}\right\}
\]
In particular, if $\E = \underline{F}$, then
\[
\Big(\cin(\M)\otimes\Gamma(\underline{\End}(\E))\Big)^n\xleftrightarrow{\text{1-1}}
\left\{\begin{array}{c}
\text{Degree}\ n\ \text{operators}\ \Psi\ \text{on}\ \cin(\M)\otimes\Gamma(\E)\ \text{which}\\
\text{are}\ \cin(\M)\text{-linear in the graded sense}
\end{array}\right\}
\]

\subsection{The Weil algebra associated to a Lie $n$-algebroid}

Let $\M$ be an $[n]$-manifold over a smooth manifold $M$ and
$\xi_1^1,\ldots,\xi_1^{r_1},\xi_2^1,\ldots,\xi_2^{r_2},\ldots,\xi_n^1,\ldots,\xi_n^{r_n}$
be its local generators over some open $U\subset M$ with degrees
$1,2,\ldots,n$, respectively. By definition, the tangent bundle $T\M$
of $\mathcal M$ is an $[n]$-manifold over $TM$ \cite{Mehta06,Mehta09},
whose local generators over $TU\subset TM$ are given by
\[
  \cin_{TU}(T\M)^0 = C^\infty(TU)\ \text{and}\ \cin_{TU}(T\M)^i =
  \xi_i^1,\ldots,\xi_i^{r_i},\diff\xi_i^1,\ldots,\diff\xi_i^{r_i}.
\]
The shifted tangent prolongation\footnote{Note that here there is a
  sign difference in the notation with \cite{Mehta06} and
  \cite{Mehta09}. $T[1]\M$ here is the same as $T[-1]\M$ in these papers.} $T[1]\M$
is an $[n+1]$-manifold over $M$, with local generators over $U$ given
by
\begin{center}
	\begin{tabular}{ |c|c| } 
		\hline
		degree 0 & $C^\infty(U)$ \\ 
		degree 1 & $\xi_1^1,\ldots,\xi_1^{r_1}, \Omega^1(U)$ \\ 
		degree 2 & $\xi_2^1,\ldots,\xi_2^{r_2}, \diff\xi_1^1,\ldots,\diff\xi_1^{r_1}$ \\ 
		\vdots & \vdots \\
		degree $n$ & $\xi_n^1,\ldots,\xi_n^{r_n}, \diff\xi_{n-1}^1,\ldots,\diff\xi_{n-1}^{r_{{n-1}}}$ \\
		degree $n+1$ & $\diff\xi_n^1,\ldots,\diff\xi_n^{r_n}$ \\
		\hline
	\end{tabular}
\end{center}
It carries a bigrading $(p,q)$, where $p$ comes from the grading of
$\M$ and $q$ is the grading of ``differential forms''. In other words,
the structure sheaf of $T[1]\M$ assigns to every coordinate domain
$(U,x^1,\ldots,x^m)$ of $M$ that trivialises $\mathcal M$, the space
\[
  \cin_U(T[1]\M) =
  \bigoplus_i\underset{\text{($i$,0)}}{\underbrace{\cin_U(\M)^i}}\left<
    \underset{\text{(0,1)}}{\underbrace{(\diff x^k)_{k=1}^m}},
    \underset{\text{(1,1)}}{\underbrace{(\diff\xi_1^k)_{k=1}^{r_1}}},\ldots,
    \underset{\text{($n$,1)}}{\underbrace{(\diff\xi_n^k)_{k=1}^{r_n}}}
  \right>.
\]

Suppose now that $(\M,\Q)$ is a Lie $n$-algebroid over $M$. Then $T[1]\M$ is an $[n+1]$-manifold,
which inherits the two commuting differentials $\ldr{\Q}$ and $\dr$ defined as follows:
\begin{itemize}
	\item $\dr\colon\cin(T[1]\M)^\bullet\to\cin(T[1]\M)^{\bullet+1}$ is defined on generators by $C^\infty(M)\ni f  \mapsto \diff f$,
	$\xi_i^j  \mapsto \diff \xi_i^j$, $\diff f\mapsto 0$ and 
	$\diff \xi_i^j  \mapsto 0$,
	and is extended to the whole algebra as a derivation of bidegree $(0,1)$.

      \item
        $\ldr{\Q}\colon\cin(T[1]\M)^\bullet\to\cin(T[1]\M)^{\bullet+1}$
        is the \textit{Lie derivative} with respect to the vector
        field $\Q$, i.e.~the graded commutator
        $\ldr{\Q} = [i_{\Q},\dr] = i_\Q\circ\dr - \dr\circ i_Q$, and
        it is a derivation of bidegree $(1,0)$. Here, $i_{\Q}$ is the bidegree
            $(1,-1)$-derivation on $T[1]\M$, which sends
            $\xi\in \cin(\M)$ to $0$, $\dr \xi$ to $\Q(\xi)$ for
            $\xi\in \cin(\M)$, and is extended to the whole algebra as
            a derivation of bidegree $(1,-1)$.
            
\end{itemize}
By checking their values on local generators, it is easy to see that
$\ldr{\Q}^2 = 0, \dr^2 = 0$ and
$[\ldr{\Q},\dr] = \ldr{\Q}\circ\dr + \dr\circ\ldr{\Q} = 0$. Hence,
$W^{p,q}(\M):=\{\text{elements of}\ \cin(T[1]\M)\ \text{of bidegree}\
(p,q) \}$ together with $\ldr{\Q}$ and $\dr$ forms a double complex.

\begin{definition}
  The Weil algebra of a Lie $n$-algebroid $(\M,\Q)$ is the
  differential graded algebra given by the total complex of
  $W^{p,q}(\M)$:
	\[
	W(\M):=\left(\bigoplus_{i\in\mathbb{Z}}\bigoplus_{i=p+q} W^{p,q}(\M),\ldr{\Q} + \dr\right).
	\]
\end{definition}
In the case of a Lie 1-algebroid $A\to M$, this is the Weil algebra
from \cite{Mehta06,Mehta09}. For the 1-algebroid case, see also
\cite{ArCr12} for an approach without the language of supergeometry.

\section{Differential graded modules}\label{modules}
This section defines the notion of a differential graded module
over a Lie $n$-algebroid $(\M,\Q)$ and gives the two fundamental
examples of modules which come canonically with any Lie $n$-algebroid,
namely the adjoint and the coadjoint modules. Note that the case of
differential graded modules over a Lie 1-algebroid $A\to M$ is studied
in detail in \cite{Mehta14}.

\subsection{The category of differential graded modules}

Let $A\to M$ be a Lie 1-algebroid. A \emph{Lie algebroid module}
\cite{Vaintrob97} over $A$ is defined as a 
sheaf $\Em$ of locally freely generated graded
$\Omega(A)$-modules over $M$ together with a map $\D\colon \Em\to\Em$ which
squares to zero and satisfies the Leibniz rule
\[
\D(\alpha\eta) = (\diff_A\alpha)\eta + (-1)^{|\alpha|}\alpha\D(\eta),
\]
for $\alpha\in\Omega(A)$ and $\eta\in\Em$. For a Lie $n$-algebroid
$(\M,\Q)$ over $M$, this is generalised to the following definitions.

\begin{definition}
  \begin{enumerate}
  	\item A \emph{left differential graded module of $(\M,\Q)$} is a
  	sheaf $\Em$ of locally freely generated left graded $\cin(\M)$-modules
  	over $M$ together with a map $\D\colon\Em\to\Em$ of degree $1$, such
  	that $\D^2=0$ and
  	\[
  	\D(\xi\eta) = \Q(\xi)\eta + (-1)^{|\xi|}\xi\D(\eta)
  	\]
  	for all $\xi\in\cin(\M)$ and $\eta\in\Em(\M)$.
      \item A \emph{right differential graded module of $(\M,\Q)$} is
        a sheaf $\Em$ of right graded modules as above together with a map $\D\colon\Em\to\Em$ of
        degree $1$, such that $\D^2=0$ and
  	\[
  	\D(\eta\xi) = \D(\eta)\xi + (-1)^{|\eta|}\eta\Q(\xi)
  	\]
  	for all $\xi\in\cin(\M)$ and $\eta\in\Em(\M)$.
      \item A \emph{differential graded bimodule of $(\M,\Q)$} is a
        sheaf $\Em$ as above together with left and right differential
        graded module structures such that the gradings and the
        differentials coincide, and the two module structures commute:
        $(\xi_1\eta)\xi_2 = \xi_1(\eta\xi_2)$ for all
        $\xi_1,\xi_2\in\cin(\M)$ and $\eta\in\Em$.
  \end{enumerate}
\end{definition}

For short we write \emph{left DG $(\M,\Q)$-module} and \emph{right DG
  $(\M,\Q)$-module}, or simply \emph{left DG $\M$-module} and
\emph{right DG $\M$-module}. The cohomology of the induced complexes
is denoted by $H_L^\bullet(\M,\Q;\Em)$ and $H_R^\bullet(\M,\Q;\Em)$,
respectively, or simply by $H_L^\bullet(\M,\Em)$ and
$H_R^\bullet(\M,\Em)$. If there is no danger of confusion, the
prefixes ``left" and ``right", as well as the subscripts ``$L$" and
``$R$", will be omitted.

\begin{definition}
  Let $(\Em_1,\D_1)$ and $(\Em_2,\D_2)$ be two left (right) differential graded
  modules over the Lie $n$-algebroids $(\M,\Q_{\M})$ and
  $(\N,\Q_{\N})$, respectively, and let $k\in\mathbb{Z}$. A degree $0$
  morphism, or simply a morphism, between $\Em_1$ and $\Em_2$ consists of a morphism of
  Lie $n$-algebroids $\phi\colon \N\to\M$ and a degree preserving map
  $\mu\colon \Em_1\to\Em_2$ which is left (right) linear:
  $\mu(\xi\eta) = \phi^\star(\xi) \mu(\eta)$, for all $\xi\in\cin(\M)$
  and $\eta\in\Em(\M)$, and commutes with the differentials $\D_1$ and
  $\D_2$.
\end{definition}

\begin{remark}
  The sheaves $\Em_1$ and $\Em_2$ in the definition above are
  equivalent to sheaves of linear functions on $\Q$-vector bundles
  over $\M$ \cite{Mehta06}.  From this point of view, it is natural
  that the definition of a morphism of differential graded modules has
  a contravariant nature.
\end{remark}

As in the case of Lie algebroids, new examples of DG $\M$-modules of
Lie $n$-algebroids are obtained by considering the usual algebraic
constructions. In the following, we describe these constructions only
for left DG modules but the case of right DG modules is treated
similarly.

\begin{definition}[Dual module]
  Given a DG $\M$-module $\Em$ with differential $\D_{\Em}$,
  one defines a DG $\M$-module structure on the dual sheaf
  $\Em^*:=\underline{\Hom}(\Em,\cin)$ with differential $\D_{\Em^*}$
  defined via the property
	\[
          \Q(\langle\psi,\eta\rangle) =
          \langle\D_{\Em^*}(\psi),\eta\rangle +
          (-1)^{|\psi|}\langle\psi,\D_{\Em}(\eta)\rangle,
	\] 
	for all $\psi\in\Em^*(\M)$ and $\eta\in\Em(\M)$, where
        $\langle\cdot\,,\cdot\rangle$ is the pairing of $\Em^*$ and
        $\Em$ \cite{Mehta06}.
\end{definition}

\begin{definition}[Tensor product]
  For DG $\M$-modules $\Em$ and $\Fm$ with operators $\D_{\Em}$ and
  $\D_{\Fm}$, the corresponding operator $\D_{\Em\otimes \Fm}$ on
  $\Em\otimes \Fm$ is uniquely characterised by the formula
	\[
          \D_{\Em\otimes \Fm}(\eta\otimes\eta') =
          \D_{\Em}(\eta)\otimes\eta' +
          (-1)^{|\eta|}\eta\otimes\D_{\Fm}(\eta'),
	\]
	for all $\eta\in \Em(\M)$ and $\eta'\in \Fm(\M)$.
\end{definition}

\begin{definition}[$\underline{\Hom}$ module]
  For DG $\M$-modules $\Em,\Fm$ with operators $\D_{\Em}$
  and $\D_{\Fm}$, the differential $\D_{\underline{\Hom}(\Em,\Fm)}$ on
  $\underline{\Hom}(\Em,\Fm)$ is defined via
	\[
	\D_{\Fm}(\psi(\eta)) = \D_{\underline{\Hom}(\Em,\Fm)}(\psi)(\eta) + (-1)^{|\psi|}\psi(\D_{\Em}(\eta)),
	\]
	for all $\psi\in
	\underline{\Hom}(\Em(\M),\Fm(\M))$ and
	$\eta\in \Em(\M)$.
\end{definition}

\begin{definition}[(Anti)symmetric powers]
	For a DG $\M$-module $\Em$ with operator $\D_{\Em}$, the corresponding operator
	$\D_{\underline{S} (\Em)}$ on ${\underline{S}^k(\Em)}$ is uniquely
	characterised by the formula
	\begin{align*}
	\D_{\underline{S} (\Em)}(\eta_1\eta_2\ldots\eta_k) = &\ \D_{\Em}(\eta_1)\eta_2\ldots\eta_k \\
	& + \eta_1\sum_{i=2}^k(-1)^{|\eta_1|+\ldots+|\eta_{i-1}|}\eta_2\ldots\D(\eta_i)\ldots\eta_k,
	\end{align*}
	for all $\eta_1,\ldots,\eta_k\in \Em(\M)$. A similar formula
         gives also the characterisation for the
        operator $\D_{\underline{A} (\Em)}$ of the antisymmetric
        powers $\underline{A}^q(\Em)$.
\end{definition}

\begin{definition}[Direct sum]
  For DG $\M$-modules $\Em,\Fm$ with operators $\D_{\Em}$
  and $\D_{\Fm}$, the differential operator $\D_{\Em\oplus \Fm}$ on
  $\Em\oplus \Fm$ is defined as
	\[
	\D_{\Em\oplus \Fm} = \D_{\Em} \oplus \D_{\Fm}.
	\]
\end{definition}

\begin{definition}[Shifts]
  For $k\in\mathbb{Z}$, the DG $\M$-module $\mathbb{R}[k]$ is defined
  as $\cin(\M)\otimes\Gamma(M\times\mathbb{R}[k])$ with differential
  given by $\Q$; here, $M\times \mathbb{R}[k]$ is the $[k]$-shift of
  the trivial line bundle over $M$, i.e. $M\times \mathbb{R}$ in
  degree $-k$ and zero otherwise.  Given now a module $\Em$ with
  differential $\D_{\Em}$, we define the shifted module
  $\Em[k]:=\Em\otimes\mathbb{R}[k]$. Due to the definition of the
  tensor module, its differential $\D[k]$ acts via
	\[
	\D_{\Em}[k](\eta\otimes1) = \D_{\Em}(\eta)\otimes1
	\]
	for all $\eta\in\Em(\M)$. Abbreviating the element
        $\eta\otimes1$ simply as $\eta$, the shifted differential
        $\D_{\Em}[k]$ coincides\footnote{Again one could choose to
          tensor with $\mathbb{R}[k]$ from the left. Then on elements
          of the form $1\otimes\eta$, the resulting differential would
          act as $\D_{\Em}[k]=(-1)^k\D_{\Em}$.} with $\D_{\Em}$.
\end{definition}

\begin{definition}
  Let $(\M,\Q_{\M})$ and $(\N,\Q_{\N})$ be Lie $n$-algebroids, and
  suppose that $\Em_1$ and $\Em_2$ are left (right) DG-modules over $\M$ and $\N$,
  respectively. A degree $k$-morphism, for $k\in\mathbb{Z}$, from
  $\Em_1$ to $\Em_2$ is defined as a left (right) degree $0$ morphism
  $\mu:\Em_1\to\Em_2[k]$; that is, a map sending elements of degree
  $i$ in $\Em_1$ to elements of degree $i+k$ in $\Em_2$, such that it
  is linear over a Lie $n$-algebroid morphism $\phi:\N\to\M$ and
  commutes with the differentials. A $k$-isomorphism is a $k$-morphism
  with an inverse.
\end{definition}

\begin{remark}
	\begin{enumerate}
		\item The inverse of a $k$-isomorphism is necessarily a $-k$-morphism.
		\item For all $k\in\mathbb{Z}$ and all DG $\M$-modules
		$\Em$, there is an obvious $k$-isomorphism
		$\Em\to \Em[k]$ over the identity on $\M$.
	\end{enumerate}
\end{remark}

Considering the special case of $\M = \N$ in the definition above
yields $k$-morphisms between DG $\M$-modules over the same Lie
$n$-algebroid. The resulting graded categories of left and right DG
$\M$-modules are denoted by $\underline{\mathbb{M}\text{od}}_L(\M,\Q)$
and $\underline{\mathbb{M}\text{od}}_R(\M,\Q)$, or simply by
$\underline{\mathbb{M}\text{od}}_L(\M)$ and
$\underline{\mathbb{M}\text{od}}_R(\M)$. The isomorphism classes of
these categories are denoted by $\underline{\text{Mod}}_L(\M,\Q)$ and
$\underline{\text{Mod}}_R(\M,\Q)$, or simply by
$\underline{\text{Mod}}_L(\M)$ and
$\underline{\text{Mod}}_R(\M)$. Again if there is no danger of
confusion, the subscripts ``$L$" and ``$R$" will be omitted.

\subsection{Adjoint and coadjoint modules}
Recall that every $[n]$-manifold $\M$ comes with the sheaf of graded
derivations $\underline{\Der}(\cin(\M))$ of $\cin(\M)$, which is
called \emph{the sheaf of vector fields over $\M$}. It is a natural
sheaf of locally freely generated graded $\cin(\M)$-modules over the
smooth manifold $M$, with left $\cin(\M)$-module structure defined by
the property $(\xi_1\mathcal{X})(\xi_2) = \xi_1\mathcal{X}(\xi_2)$ for
all $\xi_1,\xi_2\in\cin(\M)$ and
$\mathcal{X}\in\underline{\Der}(\cin(\M))$. In addition to the left
modules structure, the space of vector fields are also endowed with a
right $\cin(\M)$-module structure. The right multiplication with
functions in $\cin(\M)$ is obtained by viewing the elements of
$\underline{\Der}(\cin(\M))$ of degree $i$ as functions of bidegree
$(i,1)$ of the graded manifold $T^*[1]\M$, similarly to the shifted
tangent bundle defined before. The two module structures on
$\underline{\Der}(\cin(\M))$ are related by
$\mathcal{X}\xi = (-1)^{|\xi|(|\mathcal{X}|+1)}\xi\mathcal{X}$, for
all homogeneous $\xi\in\cin(\M)$ and
$\mathcal{X}\in\underline{\Der}(\cin(\M))$.

Suppose now that $\M$ is endowed with a homological vector field $\Q$,
i.e.~$(\M,\Q)$ is a Lie $n$-algebroid. Then the Lie derivative on the
space of vector fields $\ldr{\Q}:=[\Q,\cdot]$ is a degree 1 operator
which squares to zero and has both the left and right Leibniz
identities with respect to the left and right module structures
explained above. That is, the sheaf of vector fields over $(\M,\Q)$
has a canonical DG $\M$-bimodule structure. It is called the
\textit{adjoint module} of $\M$ and denoted by
\[
(\mathfrak{X}(\M),\ldr{\Q}).
\]

The dual module $\bigoplus_p\cin(T[1]\M)_{(p,1)}$ of 1-forms over $\M$
carries the grading obtained from the horizontal grading of the Weil
algebra -- that is, the elements of $\cin(T[1]\M)_{(p,1)}$ have degree
$p$. Its structure operator as a left DG module is given by the Lie
derivative $\ldr{\Q} = [i_\Q,\dr]$, and as a right DG module is given
by $-\ldr{\Q}$.  These left and right DG $\M$-modules are called the
\textit{coadjoint module} of $(\M,\Q)$ and denoted by
\[
  (\Omega^1(\M),\ldr{\Q})\qquad \text{and} \qquad (\Omega^1(\M),-\ldr{\Q})
\]

\subsection{Poisson Lie $n$-algebroids: coadjoint vs adjoint modules}

This section shows that a compatible pair of a homological vector
field and a Poisson bracket on an $[n]$-manifold gives rise to a
degree $-n$ map from the coadjoint to the
adjoint module which is an anti-morphism of left DG $\M$-modules and a morphism of right DG $\M$-modules.

Let $k\in\mathbb{Z}$. A degree $k$ Poisson bracket on an
$[n]$-manifold $\M$ is a degree $k$ $\mathbb{R}$-bilinear map
$\{\cdot\,,\cdot\}\colon \cin(\M)\times \cin(\M)\to \cin(\M)$,
i.e.~$|\{\xi_1,\xi_2\}| = |\xi_1| + |\xi_2| + k$, such that
$\{ \xi_1,\xi_2 \} = -(-1)^{(|\xi_1|+k)(|\xi_2|+k)}\{ \xi_2,\xi_1 \}$ and it satisfies the graded Leibniz and Jacobi identities
\[
\{\xi_1,\xi_2\xi_3\} = \{\xi_1,\xi_2\}\xi_3 + (-1)^{(|\xi_1|+k)|\xi_2|}\xi_2\{\xi_1,\xi_3\}
\] 
\[
  \{\xi_1,\{\xi_2,\xi_3\}\} = \{\{\xi_1,\xi_2\},\xi_3\} +
  (-1)^{(|\xi_1|+k)(|\xi_2|+k)}\{\xi_2,\{\xi_1,\xi_3\}\},
\]
for homogeneous elements $\xi_1,\xi_2,\xi_3\in \cin(\M)$. A morphism between two Poisson
$[n]$-manifolds $(\N,\{\cdot\,,\cdot\}_\N)$ and $(\M,\{\cdot\,,\cdot\}_\M)$ is a morphism
of $[n]$-manifolds $\mathcal{F}\colon\N\to\M$ which respects the Poisson
brackets:
$\mathcal{F}^\star\{\xi_1,\xi_2\}_\M =
\{\mathcal{F}^\ast\xi_1,\mathcal{F}^\ast\xi_2\}_\N$ for all
$\xi_1,\xi_2\in \cin(\M)$.

As is the case for ordinary Poisson
manifolds, a degree $k$ Poisson bracket on $\M$ induces a degree $k$ map
\[
\mathrm{Ham} \colon \cin(\M)\to \underline{\Der}(\cin(\M))
\]
which sends $\xi$ to its \textit{Hamiltonian vector field}
$\mathcal{X}_\xi=\{\xi\,,\cdot\}$. An $[n]$-manifold is called
\textit{symplectic} if it is equipped with a degree $k$ Poisson
bracket whose Hamiltonian vector fields generate all of
$\underline{\Der}(\cin(\M))$.

If an $[n]$-manifold $\M$ carries both a homological vector field
$\Q$ and a degree $k$ Poisson bracket $\{\cdot\,,\cdot\}$, then the
two structures are compatible if
\[
\Q\{\xi_1,\xi_2\} = \{\Q(\xi_1),\xi_2\} + (-1)^{|\xi_1|+k}\{\xi_1,\Q(\xi_2)\}
\]
  for homogeneous $\xi_1\in\cin(\M)$ and all
$\xi_2\in \cin(\M)$. Using the Hamiltonian map defined above, the
compatibility of $\Q$ and $\{\cdot\,,\cdot\}$ can be rewritten as
$\mathcal{X}_{\Q(\xi)}=[\Q,\mathcal{X}_{\xi}]$ for all
$\xi\in\cin(\M)$.

\begin{definition}
  A \emph{Poisson Lie $n$-algebroid} $(\M,\Q,\{\cdot\,,\cdot\})$ is an
  $[n]$-manifold $\M$ endowed with a compatible pair of a homological
  vector field $\Q$ and a degree $-n$ Poisson bracket
  $\{\cdot\,,\cdot\}$. If in addition the Poisson bracket is
  symplectic, then it is called a \emph{symplectic Lie
    $n$-algebroid}. A morphism of Poisson Lie $n$-algebroids is a
  morphism of the underlying $[n]$-manifolds which is also a morphism
  of Lie $n$-algebroids and a morphism of Poisson $[n]$-manifolds.
\end{definition}

A Poisson (symplectic) Lie 0-algebroid is a usual Poisson (symplectic)
manifold $M$. A Poisson Lie 1-algebroid corresponds to a Lie bialgebroid $(A,A^*)$
and a symplectic Lie 1-algebroid is again a usual Poisson manifold --
Section \ref{applications} explains this in detail. A result due to
\v{S}evera \cite{Severa05} and Roytenberg \cite{Roytenberg02} shows
that symplectic Lie 2-algebroids are in one-to-one correspondence with
Courant algebroids.

In \cite{MaXu94}, it was shown that a Lie algebroid $A$ with a linear
Poisson structure satisfies the Lie bialgebroid compatibility
condition if and only if the map $T^*A \to TA$ induced by the Poisson
bivector is a Lie algebroid morphism from $T^*A = T^*A^* \to A^*$ to
$TA \to TM$. This is now generalized to give a
characterisation of Poisson Lie $n$-algebroids.

Let $\M$ be an $[n]$-manifold equipped with a homological vector field
$\Q$ and a degree $-n$ Poisson bracket $\{\cdot\,,\cdot\}$. The Poisson
bracket on $\M$ induces a map
$\sharp\colon\Omega^1(\M)\to\mathfrak{X}(\M)[-n]$ defined one the generators via the
property
	\begin{equation}\label{eqn:sharp}
	\left( \sharp(\diff\xi_1) \right)\xi_2 = \{ \xi_1,\xi_2 \},
	\end{equation}
for all $\xi_1,\xi_2\in\cin(\M)$, and extended odd linearly by the rules
\[
\sharp(\xi_1\diff \xi_2) = (-1)^{|\xi_1|}\xi_1\sharp(\diff\xi_2)
\qquad \text{and} \qquad
\sharp(\diff\xi_1 \xi_2) = (-1)^{|\xi_2|}\sharp(\diff\xi_1) \xi_2.
\]

\begin{theorem}\label{thm_poisson}
  Let $\M$ be an $[n]$-manifold equipped with a homological vector
  field $\Q$ and a degree $-n$ Poisson bracket $\{\cdot\,,\cdot\}$. Then
  $(\M,\Q,\{\cdot\,,\cdot\})$ is a Poisson Lie $n$-algebroid if and only
  if $\sharp\colon \Omega^1(\M)\to\mathfrak{X}(\M)$ is a degree $-n$
  anti-morphism\footnote{Using convention of tensoring from the left to obtain the shift, the anti-morphism condition reads $\sharp\circ\ldr{\Q}=-\ldr{\Q}\circ\sharp=-(-1)^n\ldr{\Q}[n]\circ\sharp=-(-1)^{n(-n)}\ldr{\Q}[n]\circ\sharp=-(-1)^{n|\sharp|}\ldr{\Q}[n]\circ\sharp$.} of left DG $\M$-modules, or if and only if it is a morphism of right DG $\M$-modules,
  i.e.~$\sharp\circ\ldr{\Q}=-\ldr{\Q}\circ\sharp$.
\end{theorem}
\begin{proof}
	 From \eqref{eqn:sharp}, 
	\[
          \Big( \ldr{\Q}(\sharp(\diff\xi_1)) +
          \sharp(\ldr{\Q}(\diff\xi_1)) \Big)\xi_2 = \Q\{\xi_1,\xi_2\}
          - (-1)^{|\xi_1|-n}\{\xi_1,\Q(\xi_2)\} - \{\Q(\xi_1),\xi_2\}.
	\]
    In other words, the compatibility of $\Q$ with
    $\{\cdot\,,\cdot\}$ is equivalent to
    $\ldr{\Q}\circ\sharp = - \sharp\circ\ldr{\Q}$.
\end{proof} 
A detailed analysis of this map in the cases of Poisson Lie algebroids
of degree $n\leq2$ is given in Section \ref{morphism_of_ad*_ad_Poisson012}. The two following corollaries can be
realised as obstructions for a Lie $n$-algebroid with a Poisson
bracket to be symplectic. In particular, for $n=2$ one obtains the
corresponding results for Courant algebroids.

\begin{corollary}\label{cor_poisson}
  Let $\M$ be an $[n]$-manifold equipped with a homological vector
  field $\Q$ and a degree $-n$ Poisson bracket $\{\cdot\,,\cdot\}$. Then
  $(\M,\Q,\{\cdot\,,\cdot\})$ is symplectic if and only if $\sharp$ is
  an anti-isomorphism of left DG $\M$-modules, or if and only if it is an isomorphism of right DG $\M$-modules.
\end{corollary}

\begin{corollary}
  For any Poisson Lie $n$-algebroid $(\M,\Q,\{\cdot\,,\cdot\})$ there are
  natural degree $-n$ maps in cohomologies
  $\sharp\colon H^\bullet_L(\M,\Omega^1)\to
  H^{\bullet-n}_L(\M,\mathfrak{X})$ and $\sharp\colon H^\bullet_R(\M,\Omega^1)\to
  H^{\bullet-n}_R(\M,\mathfrak{X})$ which are isomorphisms if the
  bracket is symplectic.
\end{corollary}

\section{Representations up to homotopy}\label{ruth}
  This section generalises the notion of
representation up to homotopy of Lie algebroids from
\cite{ArCr12,GrMe10} to representations of higher Lie algebroids. Some
basic examples are given and 3-term representations of a split Lie
2-algebroid are described in detail. The adjoint and coadjoint
representations of a split Lie 2-algebroid are special examples, which
this section describes with explicit formulas for their structure
objects and their coordinate transformation. Lastly, it shows how to
define these two representations together with their objects for
general Lie $n$-algebroids for all $n$.

\subsection{The category of representations up to homotopy}

Recall that a representation up to homotopy of a Lie algebroid $A$ is
given by an $A$-module of the form
$\Omega(A,\E)=\Omega(A)\otimes\Gamma(\E)$ for a graded vector bundle
$\E$ over $M$. In the same manner, a \emph{left representation up to homotopy of
a Lie $n$-algebroid $(\M,\Q)$} is defined as a left DG $\M$-module of the
form $\cin(\M)\otimes\Gamma(\E)$ for a graded vector bundle $\E\to M$. \emph{Right representations up to homotopy} are defined similarly as right DG $\M$-modules of the form $\Gamma(\E)\otimes\cin(\M)$. In what follows, we will focus only on left representations and the prefix ``left" will be omitted for simplicity.

Following the notation from \cite{ArCr12}, we denote the category of
representations up to homotopy by $\underline{\mathbb{R}\text{ep}}^\infty(\M,Q)$,
or simply by $\underline{\mathbb{R}\text{ep}}^\infty(\M)$. The isomorphism classes of representations up to
homotopy of this category are denoted by $\underline{\text{Rep}}^\infty(\M,\Q)$,
or by $\underline{\text{Rep}}^\infty(\M)$. A representation of the form
$\E=E_0\oplus\ldots\oplus E_{k-1}$ is called \emph{$k$-term
  representation}, or simply \emph{$k$-representation}.

\begin{remark}
  Any (left) DG $\M$-module is non-canonically isomorphic to a representation
  up to homotopy of $(\M,\Q)$ \cite{Mehta14}. The proof goes as
  follows: an $\M$-module is, by definition, the sheaf of sections
  $\Gamma(\mathcal{B})$ of a vector bundle $\mathcal{B}$ over
  $\M$. The pull-back $0_{\M}^*\mathcal{B}$, where
  $0_{\M}\colon M\to\M$ is the zero embedding, is an ordinary graded
  vector bundle $\E$ over $M$ and hence splits as
  $\E=\bigoplus_i E_i[i]$. According to \cite[Theorem 2.1]{Mehta14},
  the double pull-back $\pi_{\M}^*0_{\M}^*\mathcal{B}$ is non-canonically isomorphic
  to $\mathcal{B}$ as vector bundles over $\M$, where
  $\pi_{\M}\colon \M\to M$ is the projection map. Then, as a sheaf over $M$,
  $\Gamma(\mathcal{B})$ is identified with
  $\Gamma(\pi_{\M}^*0_{\M}^*\mathcal{B})=\Gamma(\pi_{\M}^*\E)$, which
  in turn is canonically isomorphic to
  $\cin(\M)\otimes\Gamma(\E)$. 
  \end{remark}

\begin{example}[$\Q$-closed functions]
  Let $(\M,\Q)$ be a Lie $n$-algebroid and suppose $\xi\in \cin(\M)^k$
  such that
  $\Q(\xi) = 0$. Then one can construct a representation up to
  homotopy $\cin(\M)\otimes\Gamma(\E_\xi)$ of $\M$ on the graded vector bundle
  $\E_\xi=(\mathbb{R}[0]\oplus\mathbb{R}[1-k])\times M\to M$
  (i.e.~$\mathbb{R}$ in degrees 0 and $k-1$, and zero otherwise). Its
  differential $\D_\xi$ is given in components by the map
	\[
	\D_\xi = \sum_i \D_\xi^i,
	\]
	where
	\[
	\D_\xi^i\colon \cin(\M)^i\oplus \cin(\M)^{i-k+1}\to \cin(\M)^{i+1}\oplus \cin(\M)^{i-k+2}
	\]
	is defined by the formula
	\[
	\D_\xi^i(\zeta_1,\zeta_2)=(\Q(\zeta_1) + (-1)^{i-k+1}\zeta_2\xi,\Q(\zeta_2)).
      \]
	If there is an element $\xi'\in \cin(\M)^k$ which is
        $\Q$-cohomologous to $\xi$, i.e.~$\xi-\xi'=\Q(\xi'')$ for some
        $\xi''\in \cin(\M)^{k-1}$, then the representations
        $\E_\xi$ and $\E_{\xi'}$ are isomorphic via the isomorphism
        $\mu\colon \E_\xi\to \E_{\xi'}$ defined in components by
	\[
	\mu^i\colon \cin(\M)^i\oplus \cin(\M)^{i-k+1}\to \cin(\M)^{i}\oplus \cin(\M)^{i-k+1}
	\]
	given by the formula
	\[
	\mu^i(\zeta_1,\zeta_2)=(\zeta_1+\zeta_2\xi'',\zeta_2).
	\]
	Hence, one obtains a well-defined map
        $H^\bullet(\M)\to\underline{\text{Rep}}^\infty(\M)$. In particular, if $\M$ is a Lie
        algebroid, the above construction recovers Example 3.5 in
        \cite{ArCr12}.
\end{example}

\subsection{The case of (split) Lie 2-algebroids}
Fix now a split Lie $2$-algebroid $\M$, and recall that from the
analysis of Section \ref{lie2lie3}, $\M$ is given by the sum $Q[1]\oplus B^*[2]$
which forms the complex
\[
B^*\overset{\ell}{\longrightarrow} Q\overset{\rho_Q}{\longrightarrow} TM.
\]
Unravelling the data of the definition of representations up to
homotopy for the special case where $E$ is concentrated only in degree
0 yields the following characterisation.

\begin{proposition}\label{Representations_of_Lie_2-algebroids}
  A representation of the Lie $2$-algebroid $Q[1]\oplus B^*[2]$
  consists of a (non-graded) vector bundle $E$ over $M$, together with
  a $Q$-connection $\nabla$ on $E$ such that\footnote{Note that all
    the objects that appear in the following equations act via the
    generalised wedge products that were discussed before. For
    example, $\partial(\diff_\nabla e)$ and $\omega_2(\omega_2(e))$
    mean $\partial\wedge\diff_\nabla e$ and
    $\omega_2\wedge\omega_2(e)$, respectively. This is explained in
    detail in the Appendix of \cite{ArCr12}.}  :
	\begin{enumerate}[(i)]
        \item $\nabla$ is flat, i.e.~$R_\nabla = 0$ on $\Gamma(E)$,
		\item $\partial_B\circ \diff_\nabla = 0$ on $\Gamma(E)$.
	\end{enumerate}
\end{proposition}
\begin{proof}
  Let $(E,\D)$ be a representation of the Lie $2$-algebroid. Due to
  the Leibniz rule, $\D$ is completely characterised by what it does
  on $\Gamma(E)$. By definition, it sends $\Gamma(E)$ into
  $\Omega^1(Q,E)$. Using the Leibniz rule once more together with the
  definition of the homological vector field $\Q$ on $\Omega^1(Q)$,
  for all $f\in C^\infty(M)$ and all $e\in\Gamma(E)$ yields
	\[
	\D(fe) = (\rho_Q^*\diff f)\otimes e + f\D(e),
	\]
	which implies that $\D = \diff_\nabla$ for a
        $Q$-connection $\nabla$ on $\Gamma(E)$. Moreover, by definition of $\D$
        one must have $\D^2(e) = 0$ for all $e\in\Gamma(E)$. On the
        other hand, a straightforward computation yields
	\[
          \D^2(e) = \D(\diff_\nabla e) = \diff_\nabla^2 e +
          \partial_B(\diff_\nabla
          e)\in\Omega^2(Q,E)\oplus\Gamma(B\otimes E).\qedhere
	\]
\end{proof}

\begin{example}[Trivial line bundle]\label{Trivial line bundle representation example}
	The trivial line bundle $\R[0]$ over $M$ with $Q$-connection defined by
	\[
	\diff_\nabla f = \diff_Q f =\rho_Q^* \diff f
	\]
	is a representation of the Lie $2$-algebroid $Q[1]\oplus
        B^*[2]$. The operator $\D$ is given by the homological vector
        field $\Q$ and thus the cohomology induced by the
        representation is the Lie $2$-algebroid cohomology:
        $H^\bullet(\M,\R) = H^\bullet(\M)$. The shifted version of this example was used before to define general shifts of DG $\M$-modules.
\end{example}

\begin{example}\label{Trivial representation of rank k example}
  More generally, for all $k>0$, the trivial vector bundle $\R^k$ of
  rank $k$ over $M$ with $Q$-connection defined component-wise as in
  the example above becomes a representation with cohomology
  $H^\bullet(\M,\R^k)=H^\bullet(\M)\oplus\ldots\oplus H^\bullet(\M)$
  ($k$-times).
\end{example}

\begin{remark}
  Given a split Lie $n$-algebroid $A_1[1]\oplus\ldots\oplus A_n[n]$
  over a smooth manifold $M$,  with $n\geq 2$, the vector bundle
  $A_1\to M$ carries a skew-symmetric dull algebroid structure induced by the
  2-bracket and the anchor $\rho\colon A_1\to TM$ given by
  $\Q(f)=\rho^*\diff f$, for $f\in C^\infty(M)$. Hence, Proposition
  \ref{Representations_of_Lie_2-algebroids}, Example \ref{Trivial line
    bundle representation example} and Example \ref{Trivial
    representation of rank k example} can be carried over verbatim to
  the general case.
\end{remark}

A more interesting case is for representations $\E$ which are
concentrated in 3 degrees. An explicit description
of those representations is given below. The reader should note the
similarity of the following proposition with the description of 2-term
representations of Lie algebroids from \cite{ArCr12}.

\begin{proposition}\label{3-term_representations}
  A 3-term representation up to homotopy
  $(\E = E_0\oplus E_1\oplus E_2,\D)$ of $Q[1]\oplus B^*[2]$ is
  equivalent to the following data:
	\begin{enumerate}[(i)]
		\item A degree 1 map $\partial\colon \E\to \E$ such that $\partial^2 = 0$,
		\item a $Q$-connection $\nabla$ on the complex $\partial\colon E_\bullet\to E_{\bullet + 1}$,
		\item an element $\omega_2\in\Omega^2(Q,\underline{\End}^{-1}(\E))$,
              \item an element
                $\omega_3\in\Omega^3(Q,\underline{\End}^{-2}(\E))$,
                and an element
                $\phi_j\in\Gamma(B)\otimes\Omega^j(Q,\underline{\End}^{-j-1}(\E))$
                for $j=0,1$
              \end{enumerate}
              such that\footnote{In the following
                  equations, the map
                  $\partial_B\colon\Omega^1(Q)\to\Gamma(B)$ extends to
                  $\partial_B\colon \Omega^k(Q)\to\Omega^{k-1}(Q,B)$ by the
                  rule
                  $\partial_B(\tau_1\wedge\ldots\wedge\tau_k) =
                  \sum_{i=1}^k
                  (-1)^{i+1}\tau_1\wedge\ldots\wedge\hat{\tau_i}\wedge\ldots\wedge\tau_k\wedge\partial_B\tau_i$,
                  for $\tau_i\in\Omega^1(Q)$.}
	\begin{enumerate}
		\item $\partial\circ\omega_2 + \diff_\nabla^2 + \omega_2\circ\partial = 0$,
		\item
                  $\partial\circ\phi_0 + \partial_B\circ \diff_\nabla
                  + \phi_0\circ\partial = 0$,
		\item
                  $\partial\circ\omega_3 + \diff_\nabla\circ\omega_2 +
                  \omega_2\circ \diff_\nabla + \omega_3\circ\partial =
                  \langle \omega,\phi_0 \rangle$,
		\item
                  $\diff_{\overline{\nabla}}\phi_0 +
                  \partial\circ\phi_1 + \partial_B\circ\omega_2 +
                  \phi_1\circ\partial = 0$,
		\item
                  $\diff_\nabla\circ\omega_3 + \omega_2\circ\omega_2 +
                  \omega_3\circ \diff_\nabla = \langle \omega,\phi_1
                  \rangle$,
		\item
                  $\diff_{\overline{\nabla}}\phi_1 +
                  \omega_2\circ\phi_0 + \partial_B\circ\omega_3 +
                  \phi_0\circ\omega_2 = 0$,
		\item $\phi_0\circ\phi_0 + \partial_B\circ\phi_1 = 0$,
	\end{enumerate}
where $\overline{\nabla}$ is the $Q$-connection on
                $B\otimes\underline{\End}^{-j-1}(\E)$ induced by
                $\nabla$ on $B$ and $\nabla^{\underline{\End}}$ on
                $\underline{\End}(\E)$.
\end{proposition}
\begin{remark}
  \begin{enumerate}
\item  If both of the bundles $E_1$ and $E_2$ are zero, the
  equations agree with those of a 1-term representation.
\item  The equations in
  the statement can be summarised as follows:
		\[
		[\partial,\phi_0] + \partial_B\circ \diff_\nabla = 0,\qquad
		\phi_0\circ\phi_0 + \partial_B\circ\phi_1 = 0,
		\] and for all $i$:
		\[
                  [\partial,\omega_i] + [\diff_\nabla,\omega_{i-1}]
                  +\omega_2\circ\omega_{i-2} +
                  \omega_3\circ\omega_{i-3} + \ldots
                  +\omega_{i-2}\circ\omega_2 = \langle
                  \omega,\phi_{i-3} \rangle,
		\]
		\[
                  \partial_B\circ\omega_{i+2} + [\partial,\phi_{i+1}]
                  + \diff_{\overline{\nabla}}\phi_i +
                  \sum_{j\geq2}[\omega_j,\phi_{i-j+1}] = 0.
		\]

  \end{enumerate}
\end{remark}
\begin{proof}
  It is enough to check that $\D$ acts on $\Gamma(\E)$. Since $\D$ is of degree
  1, it maps each $\Gamma(E_i)$ into the direct sum
  \[
          \Gamma(E_{i+1}) \oplus
          \left(\cin(\M)^1\otimes\Gamma(E_i)\right) \oplus
          \left(\cin(\M)^2\otimes\Gamma(E_{i-1})\right) \oplus
          \left(\cin(\M)^3\otimes\Gamma(E_{i-2})\right).
	\]
	Considering the components of $\D$, this translates to the
        following three equations: 
	\[
	\D(e) = \partial(e) + d(e)\in\Gamma(E_1)\oplus\Omega^1(Q,E_0)
	\]
	for $e\in\Gamma(E_0)$,
	\[
          \D(e) = \partial(e) + d(e) + \omega_2(e) +
          \phi_0(e)\in\Gamma(E_2)\oplus\Omega^1(Q,E_1)\oplus\Omega^2(Q,E_0)\oplus\left(\Gamma(B)\otimes\Gamma(E_0)\right)
	\]
	for $e\in\Gamma(E_1)$, and
	\begin{align*}
	\D(e) = &\ d(e) + \omega_2(e) + \phi_0(e) + \omega_3(e) +\phi_1(e)\\
	& \in\Omega^1(Q,E_2)\oplus\Omega^2(Q,E_1)\oplus\left(\Gamma(B)\otimes\Gamma(E_1)\right)\\
	& \oplus\Omega^3(Q,E_0)\oplus\left(\Gamma(B)\otimes\Omega^1(Q,E_0)\right)
	\end{align*}
	for $e\in\Gamma(E_2)$. Due to Lemma
        \ref{wedge_product-operators_Correspondence_Lemma} and the
        Leibniz rule for $\D$, 
        $\partial\in\underline{\End}^1(\E)$, $d=\diff_\nabla$ where
        $\nabla$ are $Q$-connections on the vector bundles $E_i$ for
        $i = 0,1,2$,
        $\omega_i\in\Omega^i(Q,\underline{\End}^{1-i}(\E))$ for
        $i = 2,3$, and
        $\phi_i\in\Gamma(B)\otimes\Omega^i(Q,\underline{\End}^{-i-1}(\E))$
        for $i = 0,1$.

A straightforward computation and a degree count in the expansion of
the equation $\D^2=0$ shows that $(\E,\partial)$ is a complex,
$\nabla$ commutes with $\partial$, and the equations in the statement hold.
\end{proof}

\subsection{Adjoint representation of a Lie 2-algebroid}\label{adjoint}
This section shows that any split Lie $2$-algebroid
$Q[1]\oplus B^*[2]$ admits a 3-term representation up to homotopy
which is called \emph{the adjoint representation}. It is a
generalisation of the adjoint representation of a (split) Lie
$1$-algebroid studied in \cite{ArCr12}.

\begin{proposition}\label{Adjoint_representation_of_Lie_2-algebroid}
  Any split Lie $2$-algebroid $Q[1]\oplus B^*[2]$ admits a 3-term
  representation up to homotopy as follows: Choose arbitrary
  $TM$-connections on $Q$ and $B^*$ and denote both by $\nabla$. Then
  the structure objects are\footnote{Some signs are chosen
    so that the map given in
    \ref{adjoint_module_adjoint_representation_isomorphism} is an
    isomorphism for the differential of the adjoint module defined
    earlier.}
	\begin{enumerate}[(i)]
        \item the \textit{adjoint complex} $B^*[2]\to Q[1]\to TM[0]$
          with maps $-\ell$ and $\rho_Q$,
        \item the two $Q$-connections $\nabla^{\text{bas}}$ on $Q$ and
          $TM$, and the $Q$-connection $\nabla^*$ on $B^*$ given by
          the split Lie 2-algebroid,
        \item the element
          $\omega_2\in\Omega^2(Q,\Hom(Q,B^*)\oplus\Hom(TM,Q))$ defined
          by
		\[
                  \omega_2(q_1,q_2)q_3 =
                  -\omega(q_1,q_2,q_3)\in\Gamma(B^*)\ \text{and}\
                  \omega_2(q_1,q_2)X =
                  -R_\nabla^\text{bas}(q_1,q_2)X\in\Gamma(Q)
		\]
		for $q_1,q_2,q_3\in\Gamma(Q)$ and $X\in\mathfrak{X}(M)$,
        \item the element $\omega_3\in\Omega^3(Q,\Hom(TM,B^*))$
                defined by
		\[
		\omega_3(q_1,q_2,q_3)X = - (\nabla_X\omega)(q_1,q_2,q_3)\in\Gamma(B^*)
		\]
		for $q_1,q_2,q_3\in\Gamma(Q)$ and $X\in\mathfrak{X}(M)$,
		\item the element $\phi_0\in\Gamma(B)\otimes(\Hom(Q,B^*)\oplus\Hom(TM,Q))$ defined by
		\[
		\phi_0(\beta)X = \ell(\nabla_X\beta) - \nabla_X(\ell(\beta))\in\Gamma(Q)\ \text{and}\ 
		\phi_0(\beta)q = \nabla_{\rho(q)}\beta - \nabla^*_q\beta\in\Gamma(B^*)
		\]
		for $\beta\in\Gamma(B^*),q\in\Gamma(Q),X\in\mathfrak{X}(M)$,
		\item the element $\phi_1\in\Gamma(B)\otimes\Omega^1(Q,\Hom(TM,B^*))$ defined by
		\[
                  \phi_1(\beta,q)X =  \nabla_X\nabla^*_q \beta - \nabla^*_q\nabla_X
                  \beta-\nabla^*_{\nabla_X q} \beta 
                   + \nabla_{\nabla^{\rm bas}_  qX} \beta\in\Gamma(B^*)
		\]
		for $\beta\in\Gamma(B^*),q\in\Gamma(Q),X\in\mathfrak{X}(M)$.
	\end{enumerate}
\end{proposition}

The proof can be done in two ways. First, one could check explicitly
that all the conditions of a 3-representation of $Q[1]\oplus B^*[2]$
are satisfied. This is an easy but long computation and it can be found in \cite{Papantonis21}. Instead, the following
section shows that given a splitting and $TM$-connections on the
vector bundles $Q$ and $B^*$, there exists an isomorphism of sheaves
of $\cin(\M)$-modules between the adjoint module $\mathfrak{X}(\M)$
and $\cin(\M)\otimes\Gamma(TM[0]\oplus Q[1]\oplus B^*[2])$, such that the
objects defined above correspond to the differential
$\ldr{\Q}$. Another advantage of this approach is that it gives a
precise recipe for the definition and the explicit formulas for the
components of the adjoint representation of a Lie $n$-algebroid for
general $n$.

\subsection{Adjoint module vs adjoint representation}\label{adjoint_module_adjoint_representation_isomorphism}

Recall that for a split $[n]$-manifold $\M=\bigoplus E_i[i]$, the
space of vector fields over $\M$ is generated as a $\cin(\M)$-module
by two special kinds of vector fields. Namely, the degree $-i$ vector
fields $\hat{e}$ for $e\in\Gamma(E_i)$, and the family of vector
fields $\nabla^1_X \oplus \ldots \oplus \nabla^n_X$ for
$X\in\mathfrak{X}(M)$ and a choice of $TM$-connections $\nabla^i$ on
the vector bundles $E_i$.

Consider now a Lie 2-algebroid $(\M,\Q)$ together with a splitting
$\M\simeq Q[1]\oplus B^*[2]$ and a choice of $TM$-connections
$\nabla^{B^*}$ and $\nabla^Q$ on $B^*$ and $Q$, respectively. These
choices give as follows the adjoint representation $\ad_\nabla$, whose
complex is given by $TM[0]\oplus Q[1]\oplus B^*[2]$. Define a map
$\mu_\nabla\colon \cin(\M)\otimes\Gamma(TM[0]\oplus Q[1]\oplus B^*[2])\to \mathfrak{X}(\M)$ on the generators by
\[
\Gamma(B^*)\ni \beta  \mapsto \hat{\beta},\qquad
\Gamma(Q)\ni q \mapsto \hat{q},\qquad
\mathfrak{X}(M)\ni X \mapsto \nabla^{B^*}_X \oplus \nabla^Q_X
\]
and extend $\cin(\M)$-linearly to the whole space to obtain a degree-preserving
isomorphism of sheaves of $\cin(\M)$-modules. A straightforward
computation shows that
\[
\ldr{\Q}(\hat{\beta}) = \mu\left(-\ell(\beta) + \diff_{\nabla^*}\beta\right),
\]
\[
  \ldr{\Q}(\hat{q}) = \mu\left(\rho_Q(q) +
    \diff_{\nabla^{\text{bas}}}q +\omega_2(\cdot\,,\cdot)q +
    \phi_0(\cdot)q\right),
\]
\[
  \ldr{\Q}(\nabla_X^{B^*} \oplus \nabla_X^Q) = \mu\left(
    \diff_{\nabla^{\text{bas}}} X +\phi_0(\cdot)X
    +\omega_2(\cdot\,,\cdot)X + \omega_3(\cdot\,,\cdot\,,\cdot)X +
    \phi_1(\cdot\,,\cdot)X \right)
\]
and therefore, the objects in the statement of Proposition
\ref{Adjoint_representation_of_Lie_2-algebroid} define the differential $\D_{\ad_\nabla}:=\mu_\nabla^{-1}\circ\ldr{\Q}\circ\mu_\nabla$ of a
3-representation of
$Q[1]\oplus B^*[2]$, called the \textit{adjoint representation} and
denoted as $(\ad_\nabla,\D_{\ad_\nabla})$.  The adjoint representation
is hence, up to isomorphism, independent of the choice of splitting
and connections (see the following section for the precise transformations), and gives so a
well-defined class $\ad\in\underline{\text{Rep}}^\infty_L(\M)$.

Due to the result above, one can also define the \textit{coadjoint
  representation} of a Lie 2-algebroid $(\M,\Q)$ as the isomorphism
class $\ad^*\in\underline{\text{Rep}}^\infty_L(\M)$. To find an explicit representative of $\ad^*$,
suppose that $Q[1]\oplus B^*[2]$ is a splitting of $\M$, and consider
its adjoint representation $\ad_\nabla$ as above for some choice of
$TM$-connections $\nabla$ on $B^*$ and $Q$. Recall that given a
representation up to homotopy $(\E,\D)$ of a $\M$, its dual $\E^*$
becomes a representation up to homotopy with operator $\D^*$
characterised by the formula
\[
\Q(\xi\wedge\xi') = \D^*(\xi)\wedge\xi' + (-1)^{|\xi|}\xi\wedge\D(\xi'),
\]
for all $\xi\in \cin(\M)\otimes\Gamma(\E^*)$ and
$\xi'\in \cin(\M)\otimes\Gamma(\E)$.  Here,
$\wedge=\wedge_{\langle\cdot\,,\cdot\rangle}$, with
$\langle\cdot\,,\cdot\rangle$ the pairing of $\E$ with $\E^*$. Unravelling the definition of the dual for the left
representation $\ad_\nabla$, one finds that the structure differential of $\ad_\nabla^*=\cin(\M)\otimes\Gamma(B[-2]\oplus Q^*[-1] \oplus TM[0])$ is
given by the following objects:
\begin{enumerate}
\item the \textit{coadjoint complex} $T^*M\to Q^* \to B$ obtained by
  $-\rho_Q^*$ and $-\ell^*$,
	\item the $Q$-connections $\nabla$ on $B$ and $\nabla^{\text{bas},*}$ on $Q^*$ and $T^*M$,
	\item the elements
	\begin{center}
		\begin{tabular}{l l}
                  $\omega_2^*(q_1,q_2)\tau=\tau\circ\omega_2(q_1,q_2)$, & $\omega_2^*(q_1,q_2)b=-b\circ\omega_2(q_1,q_2)$, \\
                  $\phi_0^*(\beta)\tau=\tau\circ\phi_0(\beta)$, & $\phi_0^*(\beta)b=-b\circ\phi_0(\beta)$, \\
                  $\omega_3^*(q_1,q_2,q_3)b=-b\circ\omega_3(q_1,q_2,q_3)$, & $\phi_1^*(\beta,q)b=-b\circ\phi_1(\beta,q)$,
		\end{tabular}
	\end{center}
	for all $q,q_1,q_2,q_3\in\Gamma(Q),\tau\in\Gamma(Q^*),b\in\Gamma(B)$ and $\beta\in\Gamma(B^*)$.
      \end{enumerate}
  
\begin{remark}\label{Iso_coad_mod_coad_rep}
  The coadjoint representation can also be obtained from the coadjoint
  module $\Omega^1(\M)$ by the right $\cin(\M)$-module isomorphism
  $\mu^\star_\nabla\colon\Omega^1(\M)\to\Gamma(B[-2] \oplus
  Q^*[-1] \oplus T^*M[0])\otimes\cin(\M)$ which is dual to
  $\mu_\nabla\colon \cin(\M)\otimes\Gamma(TM[0]\oplus Q[1]\oplus B^*[2])\to
  \mathfrak{X}(\M)$ above. Explicitly, it is defined as the pull-back map
  $\mu^\star_\nabla(\omega)=\omega\circ\mu_\nabla$ for all $\omega\in\Omega^1(\M)$,
  whose inverse is given on the generators by 
  $\Gamma(B[-2] \oplus Q^*[-1] \oplus T^*M[0])\otimes\cin(\M)\to
  \Omega^1(\M)$,
  \[ 
  \Gamma(B)\ni b\mapsto \dr b-\diff_{\nabla^*} b, \quad \Gamma(Q^*)\ni\tau\mapsto  \dr\tau - \diff_{\nabla^*}\tau,
    \quad \text{and}\quad \Omega^1(M)\ni\theta\mapsto\theta.
  \]
  \end{remark}

\subsection{Coordinate transformation of the adjoint representation} 
The adjoint representation up to homotopy of a Lie 2-algebroid was constructed after a choice of splitting and $TM$-connections. This
section explains how the adjoint representation transforms under
different choices.

First, a morphism of 3-representations of a split Lie 2-algebroid can
be described as follows.

\begin{proposition}\label{morphism_of_3-term_representations}
  Let $(\E,\D_{\E})$ and $(\underline{F},\D_{\underline{F}})$ be
  3-term representations up to homotopy of the split Lie 2-algebroid
  $Q[1]\oplus B^*[2]$. A morphism $\mu\colon \E\to \underline{F}$ is
  equivalent to the following data:
	\begin{enumerate}[(i)]
        \item For each $i=0,1,2$, an element
          $\mu_i\in\Omega^{i}(Q,\underline{\Hom}^{-i}(\E,\underline{F}))$.
        \item An element
          $\mu^b\in\Gamma(B\otimes\underline{\Hom}^{-2}(\E,\underline{F}))$.
	\end{enumerate}
	The above objects are subject to the relations
	\begin{enumerate}
        \item
          $[\partial,\mu_i] + [\diff_\nabla,\mu_{i-1}] +
          {\displaystyle\sum_{j+k=i,i\geq2}[\omega_j,\mu_k]} = \langle
          \omega,\mu^b_{i-3} \rangle$,
		\item $[\partial,\mu^b] + [\phi_0,\mu_0] + \partial_B\circ\mu_1 = 0$,
		\item
                  $\diff_{\overline{\nabla}}\mu^b + [\phi_0,\mu_1] +
                  [\phi_1,\mu_0] + \partial_B\circ\mu_2 = 0$.
	\end{enumerate}
\end{proposition}
\begin{proof}
  As before it  suffices to check how $\mu$ acts on $\Gamma(\E)$, by the
  same arguments. Then it must be of the type
	\[
	\mu = \mu_0 + \mu_1 + \mu_2 + \mu^b,
	\]
	where
        $\mu_i\in\Omega^{i}(Q,\underline{\Hom}^{-i}(\E,\underline{F}))$
        and
        $\mu^b\in\Gamma(B)\otimes\Gamma(\underline{\Hom}^{-2}(\E,\underline{F}))$. It
        is easy to see that the three equations in the statement come from
        the expansion of
        $\mu\circ\D_{\E} = \D_{\underline{F}}\circ\mu$ when $\mu$ is
        written in terms of the components defined before.
\end{proof}

The transformation of $\ad \in \underline{\text{Rep}}^\infty(\M)$ for a fixed splitting
$Q[1]\oplus B^*[2]$ of $\M$ and different choices of $TM$-connections
is given by their difference. More precisely, let $\nabla$ and
$\nabla'$ be the two $TM$-connections. Then
 the map
$\mu=\mu_{\nabla'}^{-1}\circ \mu_\nabla\colon\ad_\nabla\to \ad_{\nabla'}$ is defined by
$\mu = \mu_0 + \mu_1 + \mu^b$, where
\begin{align*}
\mu_0 = &\ \id \\
\mu_1(q)X = &\ \nabla'_X q - \nabla_X q \\
\mu^b(\beta)X = &\ \nabla'_X \beta - \nabla_X \beta,
\end{align*}
for $X\in\mathfrak{X}(M)$, $q\in\Gamma(Q)$ and $\beta\in\Gamma(B^*)$.
The equations in Proposition \ref{morphism_of_3-term_representations}
are automatically satisfied since by construction
\[
  \D_{\ad_{\nabla'}}\circ\mu=\D_{\ad_{\nabla'}}\circ\mu_{\nabla'}^{-1}\circ\mu_\nabla=\mu_{\nabla'}^{-1}\circ\ldr{\Q}\circ\mu_\nabla=\mu_{\nabla'}^{-1}\circ\mu_\nabla\circ
  \D_{\ad_{\nabla}}=\mu\circ \D_{\ad_{\nabla}}.
  \]
  This yields the following result.

\begin{proposition}\label{Isomorphism with change of connections}
  Given two pairs of $TM$-connections on the bundles $B^*$ and $Q$,
  the isomorphism $\mu\colon\ad_\nabla\to \ad_{\nabla'}$ between the
  corresponding adjoint representations is given by
  $\mu=\id \oplus \Big( \nabla'-\nabla \Big)$.
\end{proposition}

The next step is to show how the adjoint representation transforms
after a change of splitting of the Lie 2-algebroid. Fix a Lie
2-algebroid $(\M,Q)$ over the smooth manifold $M$ and choose a
splitting $Q[1]\oplus B^*[2]$, with structure objects
$(\ell,\rho,[\cdot\,,\cdot]_1,\nabla^1,\omega^1)$ as before. Recall
that a change of splitting does not change the vector bundles $B^*$
and $Q$, and it is equivalent to a section
$\sigma\in\Omega^2(Q,B^*)$. The induced isomorphism of [2]-manifolds
over the identity on $M$ is given by:
$\mathcal{F}_\sigma^\star(\tau) = \tau$ for all $\tau\in\Gamma(Q^*)$
and
$\mathcal{F}^\star_\sigma(b) = b + \sigma^\star
b\in\Gamma(B)\oplus\Omega^2(Q)$ for all $b\in\Gamma(B)$. If
$(\ell,\rho,[\cdot\,,\cdot]_2,\nabla^2,\omega^2)$ is the structure
objects of the second splitting, then the compatibility of $\sigma$
with the homological vector fields reads the following:
\begin{itemize}
	\item The skew-symmetric dull brackets are related by $[q_1,q_2]_2 = [q_1,q_2]_1 - \ell(\sigma(q_1,q_2))$.
	\item The connections are related by
          $\nabla^2_q b = \nabla^1_q b + \partial_B\langle
          \sigma(q,\cdot),b \rangle$, or equivalently on the dual by
          $\nabla^{2*}_q \beta = \nabla^{1*}_q \beta -
          \sigma(q,\ell(\beta))$.
	\item The curvature terms are related by
          $\omega^2 = \omega^1 + \diff_{2,\nabla^1}\sigma$, where the
          operator
	\[
          \diff_{2,\nabla^1}\sigma\colon \Omega^\bullet(Q,B^*)\to\Omega^{\bullet+1}(Q,B^*)\]
        is defined by the usual Koszul formula using the dull bracket
        $[\cdot\,,\cdot]_2$ and the connection $\nabla^{1*}$.
\end{itemize}
The above equations give the following identities between the
structure data for the adjoint representations\footnote{Note that the
  two pairs of $TM$-connections are identical} $\ad_\nabla^1$ and
$\ad_\nabla^2$.

\begin{lemma}\label{Identities_for_different_splitting_of_Lie_2-algebroid}
  Let $q,q_1,q_2\in\Gamma(Q),\beta\in\Gamma(B^*)$ and
  $X\in\mathfrak{X}(M)$. Then
	\begin{enumerate}[(i)]
		\item $\ell_2 = \ell_1$ and $\rho_2 = \rho_1$.
		\item $\nabla^{2,\text{bas}}_{q_1} q_2 = \nabla^{1,\text{bas}}_{q_1} q_2 - \ell(\sigma(q_1,q_2))$
		
		$\nabla^{2,\text{bas}}_{q} X = \nabla^{1,\text{bas}}_{q} X$
		
		$\nabla^{2,*}_{q} \beta = \nabla^{1,*}_{q} \beta - \sigma(q,\ell(\beta))$.
              \item $\omega_2^2(q_1,q_2)q_3 = \omega_2^1(q_1,q_2)q_3 + \diff_{2,\nabla^1}\sigma(q_1,q_2,q_3)$\\
		$\omega_2^2(q_1,q_2)X = \omega_2^1(q_1,q_2)X +
                \nabla_X(\ell(\sigma(q_1,q_2))) -
                \ell(\sigma(q_1,\nabla_X q_2)) +
                \ell(\sigma(q_2,\nabla_X q_1))$.
		\item $\omega_3^2(q_1,q_2,q_3)X = \omega_3^1(q_1,q_2,q_3)X +  (\nabla_X(\diff_{2,\nabla^1}\sigma))(q_1,q_2,q_3)$.
		\item $\phi_0^2(\beta)q = \phi_0^1(\beta)q + \sigma(q,\ell(\beta))$\\
		$\phi_0^2(\beta)X = \phi_0^1(\beta)X$.
        \item $\phi_1^2(\beta,q)X = \phi_1^1(\beta,q)X -
                \sigma(\nabla_X q_1,\ell(\beta)) -
                \sigma(q,\ell(\nabla_X \beta)) +
                \nabla_X(\sigma(q,\ell(\beta)))$.
	\end{enumerate}
\end{lemma}

Consider now two Lie $n$-algebroids $\M_1$ and $\M_2$ over $M$, and
an isomorphism
\[
\mathcal{F}\colon(\M_1,\Q_1)\to(\M_2,\Q_2)
\]
given by the maps
$\mathcal{F}_Q\colon Q_1\to Q_2$,
$\mathcal{F}_B\colon B_1^*\to B^*_2$, and
$\mathcal{F}_0\colon\wedge^2Q_1\to B_2^*$.
Recall that a 0-morphism between two representations up to homotopy
$(\E_1,\D_1)$ and $(\E_2,\D_2)$ of $\M_1$ and $\M_2$, respectively, is
given by a degree 0 map
\[
\mu\colon \cin(\M_2)\otimes\Gamma(\E_2)\to \cin(\M_1)\otimes\Gamma(\E_1),
\]
which is $\cin(\M_2)$-linear:
$\mu(\xi\otimes e) = \mathcal{F}^\star\xi\otimes\mu(e)$ for all
$\xi\in \cin(\M_2)$ and $e\in\Gamma(\E_2)$, and makes the following
diagram commute
\[
\xymatrix{
	\cin(\M_2)\otimes\Gamma(\E_2)\ar[r]^{\mu}\ar[d]_{\D_2} & \cin(\M_1)\otimes\Gamma(\E_1)\ar[d]^{\D_1} \\
	\cin(\M_2)\otimes\Gamma(\E_2)\ar[r]_\mu &
\cin(\M_1)\otimes\Gamma(\E_1).
}
\]
The usual analysis as before implies that $\mu$ must be given by a
morphism of complexes $\mu_0\colon (\E_2,\partial_2)\to (\E_1,\partial_1)$ and
elements
\[
\mu_1\in\Omega^1(Q_1,\underline{\Hom}^{-1}(\E_2,\E_1)),
\]
\[
\mu_2\in\Omega^2(Q_1,\underline{\Hom}^{-2}(\E_2,\E_1)),
\]
\[
\mu^b\in \Gamma(B)\otimes\Gamma(\underline{\Hom}^{-2}(\E_2,\E_1)),
\]
which satisfy equations similar to the set of equations in Proposition \ref{morphism_of_3-term_representations}.

A change of splitting of the Lie 2-algebroid transforms as follows the
adjoint representation.  Since changes of choices of connections are
now fully understood, choose the same connection for both splittings
$\M_1\simeq Q[1]\oplus B^*[2]\simeq\M_2$. Suppose that $\sigma\in\Omega^2(Q,B^*)$ is the
change of splitting and denote by $\mathcal{F}_\sigma$ the induced
isomorphism of the split Lie 2-algebroids whose components are given
by
$\mathcal{F}^\star_{\sigma,Q}=\id_{Q^*}, \mathcal{F}^\star_{\sigma,B}=\id_B,
\mathcal{F}^\star_{\sigma,0}=\sigma^\star$. The composition map $\mu^\sigma:\ad_\nabla^1\to\mathfrak{X}(\M)\to\ad_{\nabla}^2$ is given in components by
\begin{align*}
\mu_0^\sigma = &\ \id \\
\mu_1^\sigma(q_1)q_2 = &\ \sigma(q_1,q_2) \\
\mu_2^\sigma(q_1,q_2)X = &\ (\nabla_X \sigma)(q_1,q_2).
\end{align*}
A similar argument as before implies that $\mu^\sigma$ is a morphism between the two adjoint
representations and therefore the following result follows.

\begin{proposition}\label{Isomorphism with change of splitting}
  Given two splittings of a Lie 2-algebroid with induced change of
  splitting $\sigma\in\Omega^2(Q,B^*)$ and a pair of $TM$-connections
  on the vector bundles $B^*$ and $Q$, the isomorphism between the
  corresponding adjoint representations is given by
  $\mu=\id\oplus\ \sigma\oplus\nabla_\cdot\sigma$.
\end{proposition}

\subsection{Adjoint representation of a Lie
  $n$-algebroid}\label{Adjoint of Lie n-algebroids}

The construction of the adjoint representation up to homotopy of a Lie
$n$-algebroid $(\M,\Q)$ for general $n$ is similar to the $n=2$
case. Specifically, choose a splitting $\M\simeq \bigoplus_{i=1}^n E_i[i]$
and $TM$-connections $\nabla^{E_i}$ on the bundles $E_i$. Then there is an
induced isomorphism of $\cin(\M)$-modules
\begin{align*}
  \mu_\nabla\colon \cin(\M)\otimes\Gamma(TM[0]\oplus E_1[1]\oplus\ldots\oplus E_n[n]) & \to \mathfrak{X}(\M),
\end{align*}
which at the level of generators is given by
\begin{align*}
  \Gamma(E_i)\ni e & \mapsto \hat{e} \quad \text{ and } \quad \mathfrak{X}(M)\ni X \mapsto \nabla^{E_n}_X \oplus \ldots \oplus \nabla^{E_1}_X.
\end{align*}
Then  $\mu$ is used to transfer $\ldr{\Q}$ from $\mathfrak{X}(\M)$
to obtain the differential
$\D_{\ad_\nabla} := \mu^{-1}\circ\ldr{\Q}\circ\mu$ on
$\cin(\M)\otimes\Gamma(TM[0]\oplus E_1[1]\oplus\ldots\oplus E_n[n]) $.

\section{Split VB-Lie $n$-algebroids}\label{VB-Lie
  n-algebroids}\label{Split VB-Lie n-algebroids}

This section gives a picture of representations up to homotopy in more
``classical'' geometric terms. That is, in terms of linear Lie
$n$-algebroid structures on double vector bundles. It introduces the
notion of split VB-Lie $n$-algebroids and explains how they
correspond to $(n+1)$-representations of Lie $n$-algebroids. In
particular, the tangent of a Lie $n$-algebroid is a VB-Lie
$n$-algebroid which is linked to the adjoint representation. The
main result in this section is a generalisation of the correspondence
between decomposed VB-algebroids and $2$-representations in
\cite{GrMe10}.

\subsection{Double vector bundles}

Recall that a double vector bundle $(D,V,F,M)$ is a commutative diagram
\[
\begin{tikzcd}
D \arrow[d,"\pi_V"'] \arrow[r,"\pi_F"] & F \arrow[d,"q_F"] \\
V \arrow[r,"q_V"']                       & M           
\end{tikzcd}
\]
such that all the arrows are vector bundle projections and the
structure maps of the bundle $D\to V$ are bundle morphisms over the
corresponding structure maps of $F\to M$ (see \cite{Mackenzie05}). This
is equivalent to the same condition holding for the structure maps of $D\to F$
over $V\to M$. The bundles $V$ and $F$ are called the side bundles of
$D$. The intersection of the kernels
$C:=\pi_V^{-1}(0^V)\cap\pi_F^{-1}(0^F)$ is the $\mathit{core}$ of $D$
and is naturally a vector bundle over $M$, with projection denoted by
$q_C\colon C\to M$. The inclusion $C\hookrightarrow D$ is denoted by
$C_m\ni c_m\mapsto\overline{c}\in
\pi_V^{-1}(0^V_m)\cap\pi_F^{-1}(0^F_m)$.

A morphism $(G_D,G_V,G_F,g)$ of two double vector bundles $(D,V,F,M)$ and $(D',V',F',M')$ is a commutative cube
\[
\begin{tikzcd}
& D \arrow[dl, "G_D"] \arrow[rr] \arrow[dd] & & F \arrow[dl, "G_F"] \arrow[dd] \\
D' \arrow[rr, crossing over] \arrow[dd] & & F' \\
& V \arrow[dl, "G_V"] \arrow[rr] & & M \arrow[dl, "g"] \\
V' \arrow[rr] & & M' \arrow[from=uu, crossing over]
\end{tikzcd}
\]
such that all the faces are vector bundle maps. 

Given a double vector bundle $(D,V,F,M)$, the space of sections of $D$
over $V$, denoted by $\Gamma_V(D)$, is generated as a
$C^\infty(V)$-module by two special types of sections, called
\textit{core} and \textit{linear} sections and denoted by
$\Gamma_V^c(D)$ and $\Gamma^l_V(D)$, respectively (see \cite{Mackenzie05}). The
core section $c^\dagger\in\Gamma_V^c(D)$ corresponding to
$c\in\Gamma(C)$ is defined as
\[
c^\dagger(v_m) = 0_{v_m}^D +_F \overline{c(m)},\, \text{ for }\, m\in M \, \text{ and }\, v_m\in V_m. 
\]
A section $\delta\in\Gamma_V(D)$ is linear over $f\in\Gamma(F)$, if $\delta\colon V\to D$ is a
vector bundle morphism  $V\to D$ over $f\colon M\to F$.

Finally, a section $\psi\in\Gamma(V^*\otimes C)$ defines a linear
section $\psi^\wedge\colon V\to D$ over the zero section $0^F\colon M\to F$ by
\[
\psi^\wedge(v_m) = 0_{v_m}^D +_F \overline{\psi(v_m)} 
\]
for all $m\in M$ and $v_m\in V_m$. This type of linear section
is called a \textit{core-linear} section. In terms of the generators
$\theta\otimes c\in\Gamma(V^*\otimes C)$, the correspondence above reads
$(\theta\otimes c)^\wedge=\ell_\theta\cdot c^\dagger$, where
$\ell_\theta$ is the linear function on $V$ associated to
$\theta\in\Gamma(V^*)$.

\begin{example}[Decomposed double vector bundle]\label{Example decomposed DVB}
  Let $V,F,C$ be vector bundles over the same manifold $M$. Set
  $D:=V\times_M F\times_M C$ with vector bundle structures
  $D=q_V^!(F\oplus C)\to V$ and $D=q_F^!(V\oplus C)\to F$. Then
  $(D,V,F,M)$ is a double vector bundle, called the decomposed double
  vector bundle with sides $V$ and $F$ and with core $C$. Its core
  sections have the form $c^\dagger\colon f_m\mapsto(0^V_m,f_m,c(m))$, for
  $m\in M,f_m\in F_m$ and $c\in\Gamma(C)$, and the space of linear
  sections $\Gamma_V^l(D)$ is naturally identified with
  $\Gamma(F)\oplus\Gamma(V^*\otimes C)$ via 
  $(f,\psi)\colon v_m\mapsto(f(m),v_m,\psi(v_m))$ where
  $\psi\in\Gamma(V^*\otimes C)$ and $f\in\Gamma(F)$. This yields the
  canonical \textit{linear horizontal lift}
  $h\colon \Gamma(F)\hookrightarrow\Gamma_V^l(D)$.
\end{example}

\begin{example}[Tangent bundle of a vector bundle]
  Given a vector bundle $q\colon E\to M$, its tangent bundle $TE$ is
  naturally a vector bundle over the manifold $E$. In addition,
  applying the tangent functor to all the structure maps of $E\to M$
  yields a vector bundle structure on $Tq\colon TE\to TM$ which is
  called the \textit{tangent prolongation} of $E$. Hence,
  $(TE,TM,E,M)$ has a natural double vector bundle structure with
  sides $TM$ and $E$. Its core is naturally identified with $E\to M$
  and the inclusion $E\hookrightarrow TE$ is given by
  $E_m\ni e_m\mapsto\left.\frac{d}{dt}\right|_{t=0}te_m\in
  T^q_{0_m^E}E$. For $e\in\Gamma(E)$, the section
  $Te\in\Gamma_{TM}^l(TE)$ is linear over $e$. The  core
  vector field $e^\dagger \in\Gamma_{TM}(TE)$ is defined by
  $e^\dagger(v_m)=T_m0^E(v_M)+_{E}\left.\frac{d}{dt}\right.\arrowvert_{t=0}te(m)$
    for $m\in M$ and $v_m\in T_ MM$ and the \textit{vertical
    lift} $e^\uparrow\in \Gamma_E(TE)=\mathfrak{X}(E)$ is the (core)
  vector field defined by the flow
  $\mathbb{R}\times E\to E,(t,e'_m)\mapsto e'_m + te(m)$. Elements of
    $\Gamma_E^l(TE)=:\mathfrak{X}^l(E)$ are called \textit{linear
      vector fields} and are equivalent to derivations
    $\delta\colon \Gamma(E)\to\Gamma(E)$ over some element in
    $\mathfrak{X}(M)$ \cite{Mackenzie05}. The linear vector field
    which corresponds to the derivation $\delta$ is written
    $X_\delta$.
\end{example}

\subsection{Linear splittings, horizontal lifts and duals}
A \textit{linear splitting} of a double vector bundle $(D,V,F,M)$ with
core $C$ is a double vector bundle embedding $\Sigma$ of the
decomposed double vector bundle $V\times_M F$ into $D$ over the
identities on $V$ and $F$. It is well-known that every double vector
bundle admits a linear splitting, see
\cite{GrRo09,delCarpio-Marek15,Pradines77} or \cite{HeJo18} for the
general case. Moreover, a linear splitting is equivalent to a
\textit{decomposition} of $D$, i.e.~to an isomorphism of double vector
bundles $S:V\times_M F\times_M C\to D$ over the identity on $V, F$ and
$C$. Given $\Sigma$, the decomposition is obtained by setting
$S(v_m,f_m,c_m)=\Sigma(v_m,f_m) +_F (0_{f_m} +_V \overline{c_m})$, and
conversely, given $S$, the splitting is defined by
$\Sigma(v_m,f_m)=S(v_m,f_,,0_m^C)$.

A linear splitting of $D$, and consequently a decomposition, is also
equivalent to a \textit{horizontal lift}, i.e.~a right splitting of
the short exact sequence
\[
0\to\Gamma(V^*\otimes C)\to \Gamma_V^l(D)\to \Gamma(F)\to 0
\]
of $C^\infty(M)$-modules.  The correspondence is given by
$\sigma_F(f)(v_m)=\Sigma(f(m),b_m)$ for $f\in\Gamma(F)$, $m\in M$ and
$b_m\in B(m)$. Note that all the previous constructions can be done
similarly if one interchanges the roles of $V$ and $F$.

\begin{example}
  For the tangent bundle $TE$ of a vector bundle $E\to M$, a linear
  splitting is equivalent to a choice of a $TM$-connection on
  $E$. Specifically, given a horizontal lift
  $\sigma\colon \mathfrak{X}(M)\to\mathfrak{X}^l(E)$, the corresponding
  connection $\nabla$ is defined by $\sigma(Y) = X_{\nabla_Y}$ for all $Y\in\mx(M)$.
\end{example}

Double vector bundles can be dualized in two ways, namely, as the dual
of $D$ either over $V$ or over $F$ \cite{Mackenzie05}. Precisely, from a double vector
bundle $(D,V,F,M)$ with core $C$, one obtains the double vector
bundles
\begin{center}
	\begin{tabular}{l r}
		\begin{tikzcd}
		D^*_V \arrow[r, "\pi_{C^*}"] \arrow[d, "\pi_V"'] & C^* \arrow[d, "q_{C^*}"] \\
		V \arrow[r, "q_V"']                               & M                       
		\end{tikzcd}
		&
		\begin{tikzcd}
		D^*_F \arrow[r, "\pi_F"] \arrow[d, "\pi_{C^*}"'] & F \arrow[d, "q_F"] \\
		C^* \arrow[r, "q_{C^*}"']                               & M                       
		\end{tikzcd}
	\end{tabular}
\end{center}
with cores $F^*$ and $V^*$, respectively.

Given a linear splitting $\Sigma\colon V\times_M F\to D$, the dual splitting
$\Sigma^*\colon V\times_M C^*\to D^*_V$ is defined by
\begin{center}
	\begin{tabular}{l c r}
          $\left\langle \Sigma^*(v_m,\gamma_m),\Sigma(v_m,f_m) \right\rangle = 0$ &
                                                                                    and $\left\langle \Sigma^*(v_m,\gamma_m), c^\dagger(v_m) \right\rangle = \left\langle \gamma_m,c(m) \right\rangle$,
	\end{tabular}
\end{center}
for all $m\in M$ and  $v_m\in V_m$, $f_m\in F_m$, $\gamma_m\in C^*_m$, $c\in\Gamma(C)$.

\subsection{VB-Lie $n$-algebroids and $(n+1)$-representations}

Suppose now that $(\underline{D},V,\A,M)$ is a double vector bundle
together with graded vector bundle decompositions
$\underline{D}=D_1[1]\oplus\ldots\oplus D_n[n]$ and
$\A=A_1[1]\oplus\ldots\oplus A_n[n]$, over $V$ and $M$, respectively, which are compatible with the
projection $\underline{D}\to\A$. This means that each of the
individual squares $(D_i,V,A_i,M)$ also forms a double vector bundle.
Schematically, this yields the following sequence of diagrams
\[
\begin{tikzcd}
& D_1[1] \arrow[ldd] \arrow[rr] \arrow[d,symbol=\oplus] &   & A_1[1] \arrow[ldd, crossing over] \arrow[d,symbol=\oplus] \\
& D_2[2] \arrow[ld] \arrow[rr] \arrow[d,symbol=\oplus]  &   & A_2[2] \arrow[ld] \arrow[d,symbol=\oplus]  \\
V \arrow[rr] & \vdots                     & M & \vdots \\
& D_n[n] \arrow[lu] \arrow[rr] \arrow[u,symbol=\oplus]  &   & A_n[n] \arrow[lu] \arrow[u,symbol=\oplus] 
\end{tikzcd}
\]
where all the ``planes" are double vector bundles. This yields that
the core of $(\underline{D},V,\A,M)$ is the graded vector
bundle $\underline{C}=C_1[1]\oplus\ldots\oplus C_n[n]$, where $C_i$ is the
core of $(D_i,V,A_i,M)$, for  $i=1,\ldots,n$.

\begin{definition}\label{VB_lien}
	The quadruple $(\underline{D},V,\A,M)$ is a \emph{(split) VB-Lie
		$n$-algebroid} if 
	\begin{enumerate}
		\item the graded vector bundle $\underline{D}\to V$ 
		is endowed with a homological vector field
		$\Q_{\underline{D}}$, 
		\item the Lie $n$-algebroid structure of $\underline{D}\to V$
		is \textit{linear}, in the sense that
		\begin{enumerate}
			\item the anchor $\rho_D\colon D_1\to TV$ is a double vector bundle morphism,
			\item the map $\partial_{D_i}$ fits into a morphism
			of double vector bundles
			$(\partial_{D_i},\id_V,\partial_{A_i},\id_M)$ between $(D_i,V,A_i,M)$ and
			$(D_{i+1},V,A_{i+1},M)$ for all $i$,
			\item the multi-brackets of
			$\underline{D}$ satisfy the following relations:
			\begin{enumerate}
				\item the $i$-bracket of $i$ linear sections is a linear section;
				\item the $i$-bracket of $i-1$ linear sections with a core section is a core section;
				\item the $i$-bracket of $i-k$ linear sections with $k$ core sections, $i\geq k \geq 2$, is zero;
				\item the $i$-bracket of $i$ core sections is zero.
			\end{enumerate} 
		\end{enumerate} 
	\end{enumerate}
\end{definition}

\begin{remark}
	\begin{enumerate}
		\item A VB-Lie $n$-algebroid structure on the double vector bundle
		$(\underline{D},V,\A,M)$ defines uniquely a Lie
		$n$-algebroid structure on $\A\to M$ as follows:
		the anchor $\rho_D\colon D_1\to TV$ is linear over the anchor
		$\rho\colon A_1\to TM$, and 
		if all
		$d_k\in\Gamma_V^l(\underline{D})$ cover $a_{k}\in\Gamma(\A)$ for
		$k=1,2,\ldots,i$, then
		$\llbracket d_1,\ldots,d_i
		\rrbracket_{\underline{D}}\in\Gamma_V^l(\underline{D})$ covers
		$\llbracket a_1,\ldots,a_i \rrbracket_{\A}\in\Gamma(\A)$. Therefore, the graded vector bundles $\underline{D}\to V$ and $\A\to M$ are
		endowed with homological vector fields $\Q_{\underline{D}}$ and
		$\Q_{\A}$ for which the bundle projection $\underline{D}\to \A$ is a
		morphism of Lie $n$-algebroids over the projection $V\to M$.
        \item  A VB-Lie 1-algebroid as in the definition
            above is just a VB-algebroid.
            \end{enumerate}
\end{remark}

\begin{example}[Tangent prolongation of a (split) Lie $n$-algebroid]
  The basic example of a split VB-Lie $n$-algebroid is obtained by
  applying the tangent functor to a split Lie $n$-algebroid
  $\A=A_1[1]\oplus\ldots\oplus A_n[n]\to M$. The double vector
  bundle is given by the diagram
	\[
	\begin{tikzcd}
	\underline{TA} \arrow[d] \arrow[r] & \A \arrow[d] \\
	TM \arrow[r]                       & M           
	\end{tikzcd}
	\]
        where the Lie $n$-algebroid structure of
        $\underline{TA}=T\A=TA_1[1]\oplus\ldots\oplus TA_n[n]$ over
        the manifold $TM$ is defined by the relations
        \begin{enumerate}
        \item $\rho_{TA}=J_M\circ T_{\rho_A}\colon TA_1\to TTM$, where
          $J_M\colon TTM\to TTM$ is the canonical involution, see e.g.~\cite{Mackenzie05},
	\item $\llbracket Ta_{k_1},\ldots,Ta_{k_i}\rrbracket = T\llbracket a_{k_1},\ldots,a_{k_i}\rrbracket$,
	\item
          $\llbracket
          Ta_{k_1},\ldots,Ta_{k_{i-1}},a_{k_i}^\dagger\rrbracket =
          \llbracket
          a_{k_1},\ldots,a_{k_{i-1}},a_{k_i}\rrbracket^\dagger$,
	\item
          $\llbracket
          Ta_{k_1},\ldots,Ta_{k_j},a_{k_{j+1}}^\dagger,\ldots,a_{k_i}^\dagger\rrbracket
          = 0$ for all $1\le j\le i-2$,
	\item $\llbracket a_{k_1}^\dagger,\ldots,a_{k_i}^\dagger\rrbracket = 0$,
\end{enumerate}	
for all sections $a_{k_j}\in\Gamma(A_{k_j})$ with pairwise distinct $k_j$ and all $i$.
\end{example}

Applying the above construction to a split Lie 2-algebroid
$Q[1]\oplus B^*[2]\to M$ with structure
$(\rho_Q,\ell,\nabla^*,\omega)$ yields  as follows the objects
$(\rho_{TQ},T\ell,T\nabla^*,T\omega)$ of the split Lie 2-algebroid
structure of $TQ[1]\oplus TB^*[2]\to TM$: The complex
$TB^*\to TQ\to TTM$ consists of the anchor of $TQ$ given by
$\rho_{TQ}=J_M\circ T\rho_Q$, and the vector bundle map
$T\ell\colon TB^*\to TQ$. The bracket of $TQ$ is defined by the
relations
	\[
	[Tq_1,Tq_2]_{TQ} = T[q_1,q_2]_Q,\qquad 
	[Tq_1,q_2^\dagger]_{TQ} = [q_1,q_2]_Q^\dagger,\qquad 
	[q_1^\dagger,q_2^\dagger]_{TQ} = 0,
	\]
	for $q_1,q_2\in\Gamma(Q)$. The $TQ$-connection
        $T\nabla^*\colon
        \Gamma_{TM}(TQ)\times\Gamma_{TM}(TB^*)\to\Gamma_{TM}(TB^*)$ is
        defined by
	\[
	(T\nabla^*)_{Tq}(T\beta) = T(\nabla^*_q\beta),\qquad
	(T\nabla^*)_{Tq}(\beta^\dagger) = (\nabla^*_q\beta)^\dagger = (T\nabla^*)_{q^\dagger}\beta,\qquad
	(T\nabla^*)_{q^\dagger}(\beta^\dagger) = 0,
	\]
	for $q\in\Gamma(Q)$ and $\beta\in\Gamma(B^*)$. Finally, the 3-form $T\omega\in\Omega^3(TQ,TB^*)$ is defined by
	\[
	(T\omega)(Tq_1,Tq_2,Tq_3) = T(\omega(q_1,q_2,q_3)),\qquad
	(T\omega)(Tq_1,Tq_2,q_3^\dagger) = \omega(q_1,q_2,q_3)^\dagger,
	\]
	\[
	(T\omega)(q_1,q_2^\dagger,q_3^\dagger) = 0 =(T\omega)(q_1^\dagger,q_2^\dagger,q_3^\dagger),
	\]
	for $q_1,q_2,q_3\in\Gamma(Q)$. 
	
        As it is shown in \cite{GrMe10}, an interesting fact about the
        tangent prolongation of a Lie algebroid is that it encodes its
        adjoint representation. The same holds for a split Lie $n$-algebroid $A_1[1]\oplus\ldots\oplus A_n[n]$, since by definition the adjoint module is exactly the space of sections of the $\Q$-vector bundle $T(A_1[1]\oplus\ldots\oplus A_n[n])\to A_1[1]\oplus\ldots\oplus A_n[n]$. The next example shows this correspondence explicitly in the case of split Lie
        2-algebroids $Q[1]\oplus B^*[2]$.

\begin{example}
  Choose two $TM$-connections on $Q$ and $B^*$, both denoted by
  $\nabla$. These choices induce the horizontal lifts
  $\Gamma(Q)\to\Gamma_{TM}^l(TQ)$ and
  $\Gamma(B^*)\to\Gamma_{TM}^l(TB^*)$, both denoted by $h$. More
  precisely, given a section $q\in\Gamma(Q)$, its lift is defined as
  $h(q) = Tq - (\nabla_{.}q)^\wedge$.
  A similar formula holds for
  $h(\beta)$ as well. Then an easy computation yields the following:
\begin{enumerate}
	\item $\rho_{TQ}(q^\dagger) = \rho(q)^\uparrow$ and $(T\ell)(\beta^\dagger) = \ell(\beta)^\uparrow$
	\item $\rho_{TQ}(h(q)) = X_{\nabla_q^{\text{bas}}}$
	\item $(T\ell)(h(\beta)) = h(\ell(\beta)) + (\nabla_.(\ell(\beta)) - \ell(\nabla_.\beta))^\wedge$
	\item $[h(q_1),h(q_2)]_{TQ} = h[q_1,q_2]_Q - R_\nabla^{\text{bas}}(q_1,q_2)^\wedge$
	\item $[h(q_1),q_2^\dagger]_{TQ} = (\nabla_{q_1}^{\text{bas}}q_2)^\dagger$
	\item $(T\nabla^*)_{h(q)}(\beta^\dagger) = (\nabla_q^*\beta)^\dagger$ 
	\item $(T\nabla^*)_{q^\dagger}(h(\beta)) = (\nabla_q^*\beta - \nabla_{\rho(q)}\beta)^\dagger$
	\item
          $(T\nabla^*)_{h(q)}(h(\beta)) =
          h(\nabla^*_q\beta) + \left(\nabla_{\nabla_\cdot q}^*\beta -
            \nabla_{\rho(\nabla_\cdot q)}\beta + \nabla^*_q\nabla_\cdot\beta -
            \nabla_\cdot\nabla_q^*\beta -
            \nabla_{[\rho(q),\cdot]}\beta\right)^\wedge$
	\item $(T\omega)(h(q_1),h(q_2),h(q_3)) = h(\omega(q_1,q_2,q_3)) + ((\nabla_\cdot\omega)(q_1,q_2,q_3))^\wedge$
	\item $(T\omega)(h(q_1),h(q_2),q_3^\dagger) = (\omega(q_1,q_2,q_3))^\dagger$.
        \end{enumerate}
      \end{example}
        
        In fact, this result is a special case of a
        correspondence between VB-Lie $n$-algebroid structures on a decomposed graded double vector bundle
        $(\underline{D},V,\A,M)$ and $(n+1)$-representations of
        $\M=\A$ on the complex
        $\E=V[0] \oplus C_1[1] \oplus \ldots \oplus C_n[n]$. In the
        general case, it is easier to give the correspondence in terms
        of the homological vector field on $\underline{D}$ and the
        dual representation on
        $\E^*=C_n^*[-n]\oplus\ldots\oplus C_1^*[-1] \oplus V^*[0]$.
    
        Suppose that $(\underline{D},V,\A,M)$ is a VB-Lie
        $n$-algebroid with homological vector fields
        $\Q_{\underline{D}}$ and $\Q_{\A}$, and choose a decomposition
        for each double vector bundle $(D_i,V,A_i,M)$\footnote{In the
          case of the tangent Lie $n$-algebroid, this corresponds to
          choosing the $TM$-connections on the vector bundles of the
          adjoint complex.}, and consequently for
        $(\underline{D},V,\A,M)$. Consider the dual
        $\underline{D}_V^*$ and recall that the spaces
        $\Gamma_V(D_i^*)$ are generated as $C^\infty(V)$-modules by
        core and linear sections. For the latter, use the
        identification
        $\Gamma_V^l(D_i^*) = \Gamma(A_i^*\otimes
        V^*)\oplus\Gamma(C_i^*)$ induced by the
        decomposition. Accordingly, the element
        $\alpha\in\Gamma(A_i^*)$ is identified with the core section
        $\pi_{\A}^{\star}(\alpha)\in\Gamma_V^c(\underline{D}^*)$.
    
        For all $\psi\in\Gamma(V^*)$, the 1-form
        $\diff\ell_\psi$ is a linear section of $T^*V\to V$ over
        $\psi$ and the anchor $\rho_{D_1}\colon D_1\to TV$ is a morphism of
        double vector bundles. This implies that the degree 1 function
        $\Q_{\underline{D}}(\ell_\psi)=\rho_{D_1}^*\diff\ell_\psi$ is
        a linear section of $\Gamma_V(\underline{D}^*)$ and thus
    \[
    \Q_{\underline{D}}(\ell_\psi)\in\Gamma_V^l(D_1^*) = \Gamma(A_1^*\otimes V^*)\oplus\Gamma(C_1^*).
    \]
    Moreover, due to the decomposition, $D_i=q_V^!(A_i\oplus C_i)$ as
    vector bundles over $V$ for all $i=1,\ldots,n$. Given
    $\gamma\in\Gamma(C_i^*)$, the function
    $\Q_{\underline{D}}(\gamma)$ lies in
    $\Gamma(\underline{S}^{i+1}(\underline{D}_V^*))$, where
    $\underline{D}_V^*=q_V^!(A_1^*\oplus C_1^*)\oplus\ldots\oplus
    q_V^!(A_n^*\oplus C_n^*)$. A direct computation shows that the
    components of $\Q_{\underline{D}}(\gamma)$ which lie in spaces
    with two or more sections of the form $\Gamma(q_V^!C_i^*)$ and
    $\Gamma(q_V^!C_j^*)$ vanish due to the bracket conditions of a
    VB-Lie $n$-algebroid. Therefore, define the representation
    $\D^*$ of $\A$ on the dual complex $\E^*$ by the equations
    \begin{center}
    	\begin{tabular}{l c r}
    		$\Q_{\underline{D}}(\ell_\psi) = \D^*(\psi)$ & and & $\Q_{\underline{D}}(\gamma) = \D^*(\gamma)$,
    	\end{tabular}
    \end{center}
    for all $\psi\in\Gamma(V^*)$ and all
    $\gamma\in\Gamma(C_i^*)$.
    
    Conversely, given a representation $\D^*$ of $\A$ on $\E^*$, the
    above equations together with
    \begin{center}
    	\begin{tabular}{l c r}
          $\Q_{\underline{D}}(q_V^*f) = \pi_{\A}^\star(\Q_{\A}(f))$ & and & $\Q_{\underline{D}}(\pi_{\A}^\star(\alpha)) = \pi_{\A}^\star(\Q_{\A}(\alpha))$
    	\end{tabular}
    \end{center}
    for all $f\in C^\infty(M)$ and $\alpha\in\Gamma(\A^*)$, define a
    VB-Lie $n$-algebroid structure on the double vector bundle
    $(\underline{D},V,\A,M)$.
    This yields the following theorem.
    \begin{theorem}
    	Let $(\underline{D},V,\A,M)$ be a decomposed graded double vector
    	bundle as above with core $\underline{C}$. There is a 1-1
    	correspondence between VB-Lie $n$-algebroid structures on
    	$(\underline{D},V,\A,M)$ and $(n+1)$-representations up to
    	homotopy of $\A$ on the complex
    	$V[0] \oplus C_1[1] \oplus \ldots \oplus C_n[n]$.
    \end{theorem}

\section{Constructions in terms of splittings}\label{applications}

This section presents in terms of splittings two of the applications
of the adjoint and coadjoint representations that were defined
before. First, there is an explicit description of the Weil algebra of
a split Lie $n$-algebroid together with its structure differentials, in
terms of vector bundles and connections similarly to
\cite{ArCr12}. Second, the map between the coadjoint and the adjoint
representations in the case of a Poisson Lie $n$-algebroid for degrees
$n\leq2$ is examined in detail.

\subsection{The Weil algebra of a split Lie $n$-algebroid}
Suppose first that $\M = Q[1]\oplus B^*[2]$ is a split Lie 2-algebroid
and consider two $TM$-connections on the vector bundles $Q$ and $B^*$,
both denoted by $\nabla$. Recall from Section
\ref{adjoint_module_adjoint_representation_isomorphism} the
(non-canonical) isomorphism of DG $\M$-modules
\[
\mathfrak{X}(\M)\cong\cin(\M)\otimes\Gamma(TM[0]\oplus Q[1]\oplus B^*[2]).
\]
 This implies that
\[
\Omega^1(\M)\cong\cin(\M)\otimes\Gamma(B[-2]\oplus Q^*[-1]\oplus T^*M[0])
\]
as (left) DG $\M$-modules, and thus the generators of the Weil algebra can be
identified with
\[
\underset{\text{($t$,0)}}{\underbrace{\cin(\M)^t}}, 
\underset{\text{(0,$u$)}}{\underbrace{\Gamma(\wedge^uT^*M)}},
\underset{\text{$(\upsilon,\upsilon)$}}{\underbrace{\Gamma(S^\upsilon Q^*)}}, 
\underset{\text{$(2w,w)$}}{\underbrace{\Gamma(\wedge^w B)}}.
\]
Using also that
$\cin(\M)^t=\bigoplus_{t=r+2s} \Gamma(\wedge^rQ^*)\otimes\Gamma(S^s
B)$, the space of $(p,q)$-forms is decomposed as
\begin{align*}
  W^{p,q}(\M,\nabla) = & \bigoplus_{\substack{p=t+v+2w \\ q=u+w+v}} \cin(\M)^t\otimes
  \Gamma\left( \wedge^uT^*M\otimes S^vQ^*\otimes \wedge^wB \right) \\
  = & \bigoplus_{\substack{p=r+2s+v+2w \\ q=u+w+v}} \Gamma\left( \wedge^uT^*M\otimes
  \wedge^rQ^*\otimes S^vQ^*\otimes \wedge^wB\otimes S^sB \right).
\end{align*}
Therefore, after a choice of splitting and $TM$-connections $\nabla$
on $Q$ and $B^*$, the total space of the Weil algebra of $\M$ can be
written as
\[
  W(\M,\nabla) = \bigoplus_{r,s,u,v,w} \Gamma\left(
    \wedge^uT^*M\otimes \wedge^rQ^*\otimes S^vQ^*\otimes
    \wedge^wB\otimes S^sB \right).
\]

The next step is to express the differentials $\ldr{\Q}$ and $\dr$ on
$W(\M,\nabla)$ in terms of the two $TM$-connections $\nabla$. For the
horizontal differential, recall that by definition the $q$-th row of
the double complex $W(\M,\nabla)$ equals the space of $q$-forms
$\Omega^q(\M)$ on $\M$ with differential given by the Lie derivative
$\ldr{\Q}$. Due to the identification of DG $\M$-modules
\[
  \Omega^q(\M)=\Omega^1(\M)\wedge\ldots\wedge\Omega^1(\M)
  =\cin(\M)\otimes\Gamma(\ad_\nabla^*\wedge\ldots\wedge\ad_\nabla^*)
\]
($q$-times) and the Leibniz identity for $\ldr{\Q}$, it follows that
the $q$-th row of $W(\M,\nabla)$ becomes the $q$-symmetric power of
the coadjoint representation $\underline{S}^q(\ad_\nabla^*)$ and
$\ldr{\Q}=\D_{\underline{S}^q(\ad_\nabla^*)}$.

The vertical differential $\dr$ is built from two 2-representations of
the tangent Lie algebroid $TM$, namely the dualization of the $TM$-representations on the graded vector bundles
$\E_{Q}=Q[0]\oplus Q[-1]$ and $\E_{B^*}=B^*[0]\oplus B^*[-1]$
whose
differentials are given by the chosen $TM$-connections
$(\id_Q,\nabla,R_\nabla)$ and $(\id_{B^*},\nabla,R_\nabla)$,
respectively. Indeed, suppose first that $\tau\in\Gamma(Q^*)$ and
$b\in\Gamma(B)$ are functions on $\M$, i.e.~$0$-forms. Then from Remark \ref{Iso_coad_mod_coad_rep}, it follows
that $\dr$ acts via
\[
  \dr\tau = \tau + \diff_{\nabla^*}\tau\qquad \text{and}\qquad \dr b
  = b + \diff_{\nabla^*}b.
\]
If now $\tau\in\Gamma(Q^*),b\in\Gamma(B)$ are 1-forms on $\M$, then
\[
  \dr\tau=\dr(\tau+\diff_{\nabla^*}\tau-\diff_{\nabla^*}\tau)
  =\dr^2\tau-\dr(\diff_{\nabla^*}\tau)=\diff_{\nabla^*}\tau-\diff_{\nabla^*}^2\tau,
\]
\[
\dr b=\dr(b+\diff_{\nabla^*}b-\diff_{\nabla^*}b)=\dr^2b-\dr(\diff_{\nabla^*}b)=\diff_{\nabla^*}b-\diff_{\nabla^*}^2b.
\]

\begin{remark}
  Note that if $B^*=0$, i.e.~$\M$ is an ordinary Lie algebroid
  $A\to M$, the above construction recovers (up to isomorphism) the
  connection version of the Weil algebra $W(A,\nabla)$ from
  \cite{ArCr11,ArCr12,Mehta09}.
\end{remark}

In the general case of a split Lie $n$-algebroid
$\M=A_1[1]\oplus\ldots\oplus A_n[n]$ with a choice of $TM$-connections
on all the bundles $A_i$, one may apply the same procedure as above to
obtain the (non-canonical) DG $\M$-module isomorphisms
\[
\mathfrak{X}(\M)\cong\cin(\M)\otimes\Gamma(TM[0]\oplus A_1[1]\oplus\ldots\oplus A_n[n])
\]
\[
\Omega^1(\M)\cong\cin(\M)\otimes\Gamma(A_n^*[-n]\oplus\ldots\oplus A_1^*[-1]\oplus T^*M[0]),
\]
and hence the identification of the generators of the Weil algebra
with
\[
\underset{\text{($t$,0)}}{\underbrace{\cin(\M)^t}}, 
\underset{\text{(0,$u$)}}{\underbrace{\Gamma(\wedge^{u}T^*M)}},
\underset{\text{$(\upsilon_1,\upsilon_1)$}}{\underbrace{\Gamma(S^{\upsilon_1} A_1^*)}},
\underset{\text{$(2\upsilon_2,\upsilon_2)$}}{\underbrace{\Gamma(\wedge^{\upsilon_2} A_2^*)}},\ldots, 
\underset{\text{$(n\upsilon_n,\upsilon_n)$}}{\underbrace{\Gamma(\wedge^{\upsilon_n} A_n^*)}}.
\] 
This then yields
\begin{align*}
  W^{p,q}(\M,\nabla) = & \bigoplus_{\substack{p=t+v_1+2v_2+\ldots \\ q=u+v_1+v_2+\ldots}} \cin(\M)^t\otimes\Gamma\left( \wedge^uT^*M\otimes S^{v_1}A_1^*\otimes \wedge^{v_2}A_2^*\otimes\ldots \right) \\
  = & \bigoplus_{\substack{p=r_1+v_1+2r_2+2v_2+\ldots \\ q=u+v_1+v_2+\ldots}} \Gamma\left( \wedge^uT^*M\otimes \wedge^{r_1}A_1^*\otimes S^{v_1}A_1^*\otimes S^{r_2}A_2^*\otimes \wedge^{v_2}A_2^*\otimes\ldots \right).
\end{align*}
Similar considerations as before imply that the $q$-th row of
$W(\M,\nabla)$ is given by $\underline{S}^q(\ad_\nabla^*)$ with
$\ldr{\Q}=\D_{\underline{S}^q(\ad_\nabla^*)}$, and that $\dr$ is built
again by the dualization of the 2-representations of $TM$ on the
graded vector bundles $\underline{E}_{A_i}=A_i[0]\oplus A_i[-1]$, for
$i=1,\ldots,n$, whose differentials are given by
$(\id_{A_i},\nabla,R_{\nabla})$.

\subsection{Poisson Lie algebroids of low degree}\label{morphism_of_ad*_ad_Poisson012}

This section describes in detail the degree $-n$ (anti-)morphism
$\sharp\colon\ad_\nabla^*\to\ad_\nabla$ of (left) right $n$-representations in the case of Poisson Lie
$n$-algebroids for $n=0,1,2$. Recall that the map $\sharp$ sends an
exact $1$-form $\diff\xi$ of the graded manifold $\M$ to the vector
field $\{\xi,\cdot\}$.

First, consider a Poisson Lie 0-algebroid, i.e.~a usual Poisson
manifold $(M,\{\cdot\,,\cdot\})$. Then the Lie $0$-algebroid is just
$M$, with a trivial homological vector field -- it can be thought of as a
trivial Lie algebroid $A=0\times M\to M$ with trivial differential $\diff_A=0$, and consequently trivial homological vector field.  The
coadjoint and adjoint representations are just the vector bundles
$T^*M[0]$ and $TM[0]$, respectively, with zero module differentials,
and the map $\sharp$ simply becomes the usual vector bundle map
induced by the Poisson bivector field that corresponds to the Poisson
bracket
\[
\sharp\colon T^*M[0] \to TM[0].
\]

Consider a Lie algebroid $A\to M$ with anchor $\rho\colon A\to TM$ and
a linear Poisson structure $\{\cdot\,,\cdot\}$, i.e.~a Lie algebroid
structure on the dual $A^*\to M$.  It is easy to see that this means
that the [1]-manifold $A[1]$ has a Poisson structure of degree $-1$.
This Poisson structure is the Schouten bracket defined on
$\Omega^\bullet(A)$ by the Lie algebroid bracket on $A^*$. Then it is
immediate that $(A[1],\diff_A,\{\cdot\,,\cdot\})$ is a Poisson Lie
$1$-algebroid if and only if $(A,A^*)$ is a Lie bialgebroid. The
latter is equivalent to $(A,\{\cdot\,,\cdot\})$ being a Poisson Lie
algebroid \cite{MaXu00}.

Let $\rho'\colon A^*\to TM$,
$\alpha\mapsto \{\alpha,\cdot\}|_{C^\infty(M)}$ and
$[\cdot\,,\cdot]_*:=\{\cdot\,,\cdot\}|_{\Omega^1(A)\times\Omega^1(A)}$
be the anchor and bracket of $A^*$, respectively. After a choice of a
$TM$-connection $\nabla$ on the vector bundle $A$, the map
$\sharp\colon \ad_\nabla^*\to\ad_\nabla$ acts via
\[
\sharp(\diff f)=\sharp_0(\diff f)= \{f,\cdot\}=-\rho'^*(\diff f)\in\Gamma(A)
\]
\[
\sharp(\beta)=\sharp_0(\beta)+\sharp_1(\cdot)\beta\in\mathfrak{X}(M)\oplus(\Gamma(A^*)\otimes\Gamma(A))
\]
for all $f\in C^\infty(M),\beta\in\Omega^1(A)$, where we identify
$\beta$ with $\dr\beta-\diff_{\nabla^*}\beta$, i.e.
$\sharp(\beta) = \sharp(\dr\beta) - \sharp(\diff_{\nabla^*}\beta) =
\{\beta,\cdot\}-\sharp(\diff_{\nabla^*}\beta)$. Computing how these
act on $\alpha\in\Omega^1(M)$ and $g\in C^\infty(M)$, viewed as
functions of the graded manifold $A[1]$, gives the components of $\sharp(\beta)$: From the right-hand-side of the equation we obtain
\[
\left(\sharp_0(\beta)+\sharp_1(\cdot)\beta\right)g = \sharp_0(\beta)g\in C^\infty(M)
\]
while from the left-hand-side we obtain
\[
\sharp(\beta)g = \sharp(\dr \beta - \diff_{\nabla^*}\beta)g  = 
\{\beta,g\} - \sharp(\diff_{\nabla^*}\beta)g =  \rho'(\beta)g.
\]
From this, it follows that $\sharp_0(\beta) = \rho'(\beta)$. Using now this, the right-hand-side gives
\[
\left(\sharp_0(\beta)+\sharp_1(\cdot)\beta\right)\alpha =  \nabla^*_{\rho'(\beta)}\alpha + \left(\sharp_1(\cdot)\beta\right)\alpha\in\Gamma(A^*)\oplus\Gamma(A^*)
\]
while the left-hand-side gives
\[
\sharp(\beta)\alpha = \sharp(\dr \beta - \diff_{\nabla^*}\beta)\alpha  = 
\{\beta,\alpha\} - \sharp(\diff_{\nabla^*}\beta)\alpha =
[\beta,\alpha]_* + \nabla^*_{\rho'(\alpha)}\beta = (\nabla^*)^{\text{bas}}_\beta\alpha.
\]
This implies that $(\sharp_1(\cdot)\beta)\alpha = (\nabla^*)^{\text{bas}}_\beta\alpha - \nabla^*_{\rho'(\alpha)}\beta$ and thus $\sharp$ consists
of the ($-1$)-chain map $\sharp_0$ given by the anti-commutative diagram
\[
\begin{tikzcd}
T^*M[0] \arrow[r, "-\rho^*"] \arrow[d, "-\rho'^*"] & A^*[-1] \arrow[d, "\rho'"] \\
A[1] \arrow[r, "\rho"]                            & TM  [0]                  
\end{tikzcd}
\]
together with
$\sharp_1(a)\beta = \langle (\nabla^*)^{\text{bas}}_\beta(\cdot) -
\nabla^*_{\rho'(\beta)}(\cdot),a
\rangle\in\Gamma(A^{**})\simeq\Gamma(A)$, for all
$\beta\in\Gamma(A^*)$ and $a\in\Gamma(A)$.

By Theorem
\ref{thm_poisson}, $\sharp$ is an (anti-)morphism of $2$-representations if
and only $(A[1],\diff_A,\{\cdot\,,\cdot\})$ is a Poisson Lie
$1$-algebroid. Hence, $\sharp$ is an (anti-)morphism of $2$-representations if
and only if $(A,A^*)$ is a Lie bialgebroid.  Similarly,
\cite{GrJoMaMe18} shows that $\ad_\nabla^*$ and $\ad_\nabla$ form a
\emph{matched pair} if and only if $(A,A^*)$ is a Lie bialgebroid.

Note that $(A,\{\cdot\,,\cdot\})$ is a Poisson Lie algebroid if the
induced vector bundle morphism $\sharp\colon T^* A\to TA$ over $A$ is
a VB-algebroid morphism over $\rho'\colon A^*\to TM$
\cite{MaXu00}. Then the fact that
$\sharp\colon \ad_\nabla^*\to\ad_\nabla$ is an (anti-)morphism of
$2$-representations follows immediately \cite{DrJoOr15}, since
$\ad_\nabla^*$ and $\ad_\nabla$ are equivalent to decompositions of
the VB-algebroids $(T^*A\to A^*, A\to M)$ and $(TA\to TM, A\to M)$,
respectively.

\medskip Now consider the case of 2-algebroids. First recall that a
symplectic Lie 2-algebroid over a point, that is, a Courant algebroid
over a point, is a usual Lie algebra $(\mathfrak{g},[\cdot\,,\cdot])$
together with a non-degenerate pairing
$\langle \cdot\,,\cdot \rangle\colon
\mathfrak{g}\times\mathfrak{g}\to\mathfrak{g}$, such that
\[
  \langle [x,y],z \rangle + \langle y,[x,z] \rangle = 0\ \text{for
    all}\ x,y,z\in\mathfrak{g}.
\]
Using the adjoint and coadjoint representations
$\ad\colon\mathfrak g\to \End(\mathfrak g)$, $x\mapsto [x,\cdot]$, and
$\ad^*\colon \mathfrak g\to\End(\mathfrak g^*)$, $x\mapsto -\ad(x)^*$, and
denoting the canonical linear isomorphism induced by the pairing by
$P\colon \mathfrak{g}\to\mathfrak{g^*}$, the equation above reads
\[
P(\ad(x)y) = \ad^*(x)P(y)\ \text{for all}\ x,y\in\mathfrak{g}.
\]
In other words, this condition is precisely what is needed to turn the
vector space isomorphism $P$ into an isomorphism of Lie algebra
representations between $\ad$ and $\ad^*$. In fact, the map
$\sharp\colon \ad^*\to\ad$ for Poisson Lie 2-algebroids is
a direct generalisation of this construction.

Let $B\to M$ be a usual Lie algebroid with a 2-term representation
$(\nabla^Q,\nabla^{Q^*},R)$ on a complex $\partial_Q\colon Q^*\to
Q$. The representation is called \textit{self dual} \cite{Jotz18b} if
it equals its dual, i.e.~$\partial_Q=\partial_Q^*$, the connections
$\nabla^Q$ and $\nabla^{Q^*}$ are dual to each other, and
$R^*=-R\in\Omega^2(B,\Hom(Q,Q^*))$,
i.e.~$R\in\Omega^2(B,\wedge^2Q^*)$.  \cite{Jotz18b} further shows that
Poisson brackets $\{\cdot\,,\cdot\}$ on a split Lie 2-algebroid
$Q[1]\oplus B^*[2]$ correspond to self dual 2-representations of $B$
on $Q^*[1]\oplus Q[0]$ as follows: the bundle map $\partial_Q\colon
Q^*\to Q$ is
$\tau\mapsto\{ \tau,\cdot \}|_{\Omega^1(Q)}$, the anchor $\rho_B\colon B\to TM$ is
$b\mapsto\{ b,\cdot \}|_{C^\infty(M)}$, the $B$-connection on $Q^*$ is given by
$\nabla^{Q^*}_b\tau=\{b,\tau\}$, and the 2-form $R$ and the Lie bracket of
$B$ are defined by
$\{b_1,b_2\} = [b_1,b_2] - R(b_1,b_2)\in\Gamma(B)\oplus\Omega^2(Q)$.

Fix now a Poisson Lie 2-algebroid $(\M,\Q,\{\cdot\,,\cdot\})$ together
with a choice of a splitting $Q[1]\oplus B^*[2]$ for $\M$, a pair of
$TM$-connections on $B^*$ and $Q$, and consider the representations
$\ad_\nabla$ and $\ad_\nabla^*$. Similarly as before, we have that
\[
\sharp(\diff f) = \sharp_0(\diff f) = \{f,\cdot\} = -\rho_B^*(\diff f)\in\Gamma(B^*)
\]
\[
  \sharp(\tau) = \sharp_0(\tau) +
  \sharp_1(\cdot)\tau\in\Gamma(Q)\oplus(\Omega^1(Q)\otimes\Gamma(B^*))
\]
\[
  \sharp(b) = \sharp_0(b) + \sharp_1(\cdot)b + \sharp_2(\cdot\,,\cdot)b +
  \sharp^b(\cdot)b\in\mathfrak{X}(M)\oplus\Omega^1(Q,Q)\oplus\Omega^2(Q,B^*)\oplus(\Gamma(B)\otimes\Gamma(B^*))
\]
for $f\in C^\infty(M),\tau\in\Gamma(Q^*),b\in\Gamma(B)$, where we
identify $\tau$ with $\dr\tau-\diff_{\nabla^*}\tau$ and $b$ with
$\dr b-\diff_{\nabla^*}b$. Then the map
$\sharp\colon \ad^*_\nabla\to\ad_\nabla$ consists of the ($-2$)-chain
map given by the anti-commutative diagram
\[
\begin{tikzcd}
  T^*M[0] \arrow[r, "-\rho_Q^*"] \arrow[d, "-\rho_B^*"] & Q^*[-1]
  \arrow[r, "-\partial_B"] \arrow[d, "\partial_Q"]
  & B[-2] \arrow[d, "\rho_B"] \\
  B^*[2] \arrow[r, "-\partial_B^*"] & Q[1] \arrow[r, "\rho_Q"] & TM[0]
\end{tikzcd}
\]
and the elements
\[
\sharp_1(q)\tau = \langle \tau, \nabla^Q_\cdot q - \nabla_{\rho_B(\cdot)}q \rangle\in\Gamma(B^*)
\]
\[
\sharp_1(q)b = \nabla_b^Q q - \nabla_{\rho_B(b)}q\in\Gamma(Q)
\]
for $q\in\Gamma(Q),\tau\in\Gamma(Q^*),b\in\Gamma(B)$,
\[
\sharp_2(q_1,q_2)b = - \langle R(b,\cdot)q_1,q_2 \rangle\in\Gamma(B^*)
\]
for $q_1,q_1\in\Gamma(Q),b\in\Gamma(B)$, where $R$ is the
component that comes from the self-dual 2-representation of
$B$ from the Poisson structure,
\[
  \sharp^b(\beta)b = \langle \beta,\nabla^{\text{bas}}_b(\cdot) -
  \nabla^*_{\rho_B(b)}(\cdot) \rangle \in\Gamma(B^*)
\]
for $\beta\in\Gamma(B^*),b\in\Gamma(B)$. 

Suppose now that the split Lie 2-algebroid is symplectic, i.e.~that it
is of the form $E[1]\oplus T^*M[2]$ for a Courant algebroid $E\to
M$. The only thing that is left from the construction in the Example
\ref{Split_symplectic_Lie_2-algebroid_example} is a choice of a
$TM$-connection on $TM$, and hence on the dual $T^*M$. The (anti-)isomorphism
$\sharp\colon \ad_\nabla^*\to\ad_\nabla$ consists of the (-2)-chain
map of the anti-commutative diagram
\[
\begin{tikzcd}
  T^*M[0] \arrow[r, "-\rho^*"] \arrow[d, "-\id"] & E^*[-1] \arrow[r, "\rho"] \arrow[d, "P^{-1}"] & TM[-2] \arrow[d, "\id"] \\
  T^*M[2] \arrow[r, "\rho^*"] & E[1] \arrow[r, "\rho"] & TM [0]
\end{tikzcd}
\]
where $P\colon E\overset{\sim}{\to} E^*$ is the pairing, and the elements
$\langle \sharp_2(e_1,e_2)X,Y \rangle = - \langle
R_\nabla(X,Y)e_1,e_2 \rangle$ and
$\langle \sharp^b(\alpha)X, Y \rangle = - \langle
\alpha,T_\nabla(X,Y) \rangle$. Its inverse consists of the
2-chain map given by the anti-commutative diagram
\[
\begin{tikzcd}
T^*M[2] \arrow[r, "\rho^*"] \arrow[d, "-\id"] & E[1] \arrow[r, "\rho"] \arrow[d, "P"] & TM[0] \arrow[d, "\id"] \\
T^*M[0] \arrow[r, "-\rho^*"]                   & E^*[-1] \arrow[r, "\rho"]                      & TM  [-2]              
\end{tikzcd}
\]
and the elements
$\langle \sharp^{-1}_2(e_1,e_2)X,Y \rangle = - \langle
R_\nabla(X,Y)e_1,e_2 \rangle$ and
$\langle (\sharp^{-1})^b(\alpha)X, Y \rangle = - \langle
\alpha,T_\nabla(X,Y) \rangle$. In other words, $\sharp^2=\id$. If the
connection on $TM$ is torsion-free, then the terms $\sharp^b$ and
$(\sharp^{-1})^b$ vanish, as well. In particular, if the base manifold
$M$ is just a point, then the bundles $TM$ and $T^*M$, and the
elements $\sharp_2$ and $\sharp^{-1}_2$ are zero. Therefore, the map
$\ad^*_\nabla\to\ad_\nabla$ reduces to the linear isomorphism of the pairing and agrees with the one explained above.

\def\cprime{$'$} \def\polhk#1{\setbox0=\hbox{#1}{\ooalign{\hidewidth
  \lower1.5ex\hbox{`}\hidewidth\crcr\unhbox0}}} \def\cprime{$'$}
  \def\cprime{$'$}

	
\end{document}